\documentclass[a4paper,10pt]{article}
\usepackage[utf8]{inputenc}
\usepackage{amsmath}
\usepackage{amsfonts}
\usepackage{amssymb,amsthm}
\usepackage{todonotes} 
\usepackage{algorithm}
\usepackage{algorithmic}
\usepackage{hyperref}

\newtheorem{theorem}{Theorem}
\newtheorem{proposition}{Proposition}
\newtheorem{corollary}{Corollary}
\newtheorem{remark}{Remark}
\newtheorem{definition}{Definition}

\newcommand{\R}{\mathbb{R}}
\newcommand{\Op}{\mathcal{A}}
\newcommand{\Ap}{\overline{\mathcal{A}}}
\newcommand{\y}{\mathrm{y}} 
\newcommand{\ypure}{y} 
\newcommand{\xd}{x^\dagger} 
\newcommand{\KS}{\mathcal{K}} 
\newcommand{\RS}{\mathcal{R}}
\newcommand{\KRS}{\mathcal{KR}}
\newcommand{\X}{\mathcal{X}}
\newcommand{\Q}{\mathcal{Q}} 
\newcommand{\be}{\begin{equation}}
\newcommand{\ee}{\end{equation}}

\newcommand{\PS}{\mathcal{P}}
\newcommand{\spa}{{\rm span}} 

\newcommand{\POP}[1]{P_{#1}}
\newcommand{\pone}{{s}}
\newcommand{\ptwo}{{t}}
\newcommand{\pold}{p_{\rm old}} 
 
\newcommand{\MAXIT}{{\rm MAXIT}} 
\newcommand{\grdiag}{\kappa}
\newcommand{\factor}{\rho}
\newcommand{\vg}{\vec{g}}
\newcommand{\vgm}{g}
\DeclareMathOperator*{\argmin}{argmin}
\newcommand{\skp}[1]{\langle #1 \rangle } % Scalar Product 

\newcommand{\tmat}[3]{\begin{bmatrix}#1_1 & #2_2 &#3_3 & & & & &\\ 
 #2_2  & #1_2 & #2_3 & & & & &\\
#3_3 & #2_3 & #1_3 & #2_4 & #3_5 & & &\\ 
& & #2_4 & #1_4 & #2_5 & & &\\ 
& & #3_5 & #2_5 & #1_5 & #2_6 & #3_7 &\\ 
& &  &   &  #2_6  & #1_6 & #2_7 &\\ 
& &  &   &  #3_7  & #2_7 &  #1_7 & \ldots \\
 & &  &   &    &  \ldots&  \ldots  &  \ldots \\
  \end{bmatrix} }

\title{A rational conjugate gradient method for linear 
ill-conditioned problems}
\author{Stefan Kindermann\footnotemark[2] and Werner Zellinger\footnotemark[3]}

\begin{document}
\renewcommand{\thefootnote}{\fnsymbol{footnote}}
\footnotetext[2]{Industrial Mathematics Institute, Johannes Kepler University Linz, Austria. 
({\tt kindermann@indmath.uni-linz.ac.at}).}
\footnotetext[3]{Johann Radon Institute for Computational and Applied Mathematics (RICAM), 
Austrian Academy of Science, Linz, Austria. 
({\tt Werner.Zellinger@oeaw.ac.at})} 
\maketitle
\begin{abstract}
We consider linear ill-conditioned operator equations in a Hilbert space setting. 
Motivated by the aggregation method, we consider approximate solutions constructed  
from linear combinations of Tikhonov regularization, which amounts to finding 
solutions in a rational Krylov space. By mixing these with usual Krylov spaces, we
consider least-squares problem in these mixed rational  spaces. Applying the Arnoldi method leads
to a sparse, pentadiagonal representation of the forward operator, and we introduce 
the Lanczos method for solving the least-squares problem by factorizing this matrix. 
Finally, we present an equivalent conjugate-gradient-type method that does not rely 
on explicit orthogonalization but uses short-term recursions and Tikhonov regularization
in each second step. We illustrate the convergence and regularization properties by some numerical examples.
\end{abstract} 
\renewcommand{\thefootnote}{\arabic{footnote}}
\section{Introduction} 
The setting of this article are linear ill-posed problems stated in Hilbert spaces. 
That is, given a compact forward operator $A :X \to Y$ between Hilbert spaces $X,Y$ and 
data $y \in Y$, a standard idea to find (generalized) solutions $x$ of $A x = y$ is 
the least-squares approach: 
\[ \min_{x \in X} \|A x -y \|. \]
In this work, we are particular interested in the 
case that $A$ represents an ill-posed or ill-conditioned forward operator such that 
direct method usually lead to useless solutions, but regularization has to be 
employed. One of the most popular regularization methods in this case is Tikhonov regularization 
that calculates approximate solution to the forward problem by 
\begin{equation}\label{defTik}
 x_{\alpha}:= (A^*A  + \alpha I )^{-1}A^*y, 
\end{equation} 
with $\alpha>0$ being a regularization parameter that has to be chosen 
appropriately. Since this involves solving a linear system, especially in high-dimensional 
cases, iterative methods are the state-of-the-art, 
 for instance, highly popular are Krylov-space methods (see, e.g., \cite{Regtools}), 
which in this case lead to applying the conjugate gradient (CG) method \cite{cg} to the normal equations
(CGNE)
$A^* A x = A^* y$ \cite{EHN,Hankebook}. By using a discrepancy principle as stopping 
criterion, this CGNE method acts as a regularization method. 

There are two main sources of inspiration for our work. The first one is the
so-called aggregation method  \cite{ChenPere15}, which improves Tikhonov regularization 
by constructing linear combinations of several $x_{\alpha_i}$ in \eqref{defTik} 
and minimizing the least-squares functional over such combinations. 
This has, e.g.,  been successfully used in combination with heuristic parameter choice rules 
\cite{PeKiPi}, inverse problems in geophysics \cite{Tka}, and in particular in 
domain adaption in learning \cite{agglearning,Pereveryzevbook}.

It is not difficult to rephrase  this approach as least-squares minimization 
over a rational Krylov space. We discuss this method in Section~\ref{sec:agg}.
We show below that the resulting residual is always smaller than that for each 
individual Tikhonov regularization and also smaller than that of the CGNE method
with the same number of steps. However, the aggregation method is non-iterative, and 
by adding new $x_{\alpha_i}$, one cannot make efficient use of 
the previous solutions. The main point of this article is to state a
similar (but not equivalent) iterative method that also uses rational Krylov spaces 
and is practically as efficient as the aggregation method. 

To do so, we employ the second source of inspiration, namely the use of mixed 
rational Krylov spaces that have been investigated by Prani\'c and Reichel~\cite{prarei}. 
It has been shown there that by mixing the rational spaces of the aggregation method 
and the usual Krylov spaces of the CGNE method, one finds  short recurrences of the residuals. 
We introduce the corresponding spaces 
in Section~\ref{sec:agg} and prove a pentadiagonal representation of the 
forward operator in the mixed rational space in Section~\ref{sec:arnoldi}. 

As a consequence of the sparse representation, one can find an iterative solution 
method for the least-squares problem in the mixed rational space by using 
Arnoldi orthogonalization and essentially an $LU$-factorization of the sparse 
matrix. This leads to the Lanczos method defined in Section~\ref{sec:lanczos}. 
The drawback of this method is the corresponding memory requirement since all 
orthogonal basis vectors have to be saved. We therefore propose an equivalent 
iterative method that does not need this but is based on short recurrence
relations and Tikhonov regularization. The corresponding algorithm, 
called rational CG method, is described in Section~\ref{sec:ratcg} and is considered the 
main contribution of this work. We prove that the method computes the same sequence 
of approximate solutions as the Lanczos method, which are minimizers of the least-squares 
problem in the mixed rational Krylov spaces. 

Let us remark that we do not prove regularization properties of the Lanczos and rational CG algorithms
but leave this for future work. Note that the proof that the CGNE method is a regularization 
method 
 when combined with the discrepancy principle 
(see, e.g., the classical monograph \cite{Hankebook} or \cite{Nem1,Nem2}) 
 is quite involved, and we expect the same 
to be true for the rational counterpart. 

Finally, let us mention some related work. As far as the authors know, the proposed rational 
CG method is original, however, there exist some related methods in the literature. 
Ruhe \cite{Ruhe} has introduced rational Krylov spaces for eigenvalues computations \cite{Ruhe}. 
G\"uttel \cite{Guttel}  has used a quite similar methods as our Lanczos algorithms for solving (not necessarily symmetric) linear problems, but without   
mixing the spaces.
Grimm \cite{Grimm} has studied regularization properties of such rational CG method, but he uses only constant 
regularization parameters, whereas we allow varying ones. Compared to Grimm, this makes the algorithm 
more complicated (we get a three-term recursion in the odd steps), but our method is more flexible and is comparable 
to the aggregation method. However,  the 
derivation and analysis is also more involved.
 Prani\'{c} et al.~\cite{prareirod} proposed a 
Lanczos-type method for non-symmetric linear equation that is essentially identical to the 
Lanczos method in this article, which  can thus be considered a specialization of  
that in \cite{prareirod} to the symmetric case. 

What is the main novelty in our work compared to \cite{prareirod,Guttel,Grimm} is the rational CG method 
in Section~\ref{sec:ratcg}, which we regard as a nontrivial extension of the cited work. It is simpler 
to implement than the Lanczos method and seems to be more robust  numerically.

\section{The aggregation method and rational Krylov spaces}\label{sec:agg}
For  notational reasons we introduce the system matrix (or operator) for the 
normal equations and the corresponding right-hand side:
\be\label{defop} \Op:= A^*A \qquad \y := A^*y, \qquad \Ap : = AA^* \, .\ee 
In the sequel we denote the Hilbert-space inner products in $X$ or $Y$ by $\skp{.,.}$. 
We consider  a sequence of pairwise disjoint nonnegative regularization parameter 
\[ \alpha_1, \ldots \alpha_n,\ldots \qquad \alpha_i >0,  \qquad \alpha_i \not = \alpha_j \quad \text{for } i\not = j.\]
For any such $\alpha_i$, we define the associated Tikhonov regularization 
$x_{\alpha_i}$ defined by \eqref{defTik}.
The method of aggregation by a linear functional strategy of Chen and Pere\-very\-zev~\cite{ChenPere15}
is based on improving the approximate solutions $x_{\alpha_i}$ by 
building finite linear combination of the already computed solutions $(x_{\alpha_i})_{i=1}^N$ 
and minimizing  $\| \sum_{i=1}^N c_i x_{\alpha_i} - x^*\|_X $ 
over the coefficients $c_i$ and where $x^*$ is an approximation of the true solution $\xd$ using  
a linear functional strategy. Since this requires an approximation of $\xd$, which is not available, 
another variant of this idea is to minimize the residual instead. 
This has been proposed in \cite{aggregation}, \cite[Section~4.4]{Pereveryzevbook} in learning theory. 
It amounts to minimizing the least-squares functional \mbox{$\| A \sum_{i=1}^N c_i x_{\alpha_i}  -y\|$}
over the coefficients $c_i$ yielding an approximation $\sum_{i=1}^N c_i x_{\alpha_i}$ with minimal 
residual under all such linear combinations. A similar idea 
is used in Anderson acceleration \cite{Anderson} but with convex combinations of $x_{\alpha_i}$ instead of linear ones.  

The motivation for this article is the simple observation that the linear combination of $x_{\alpha_i}$ is 
an element of a rational Krylov space, and thus the aggregation strategy is  a generalization of 
classical Krylov-space minimization. 
Let us define the usual Krylov space of dimension $n$ for the above least-squares problem as it is used in 
the CGNE method (using the notation~\eqref{defop}):
\[ \KS^n := \spa\{\y,\Op \y, \Op^2 \y, \ldots, \Op^{n-1} \y \}.  \]
We define 
\begin{equation}\label{eq:xks} x_{\KS,n}:= \argmin_{x \in \KS^{n}} \|  A x - y\|.  \end{equation}
It is well known \cite{Hankebook,EHN} that $x_{\KS,n}$ can be calculated  recursively by  the $n$th step of the 
CGNE method.

In contrast, the acceleration method uses elements in the following rational Krylov space:
Define 
\[ f_{\alpha}(\lambda) = \frac{1}{\lambda + \alpha}. \]
Then, the rational Krylov space of dimension $n$ is given as 
\[ \RS^{n} = \spa\{f_{\alpha_1}(\Op) \y, f_{\alpha_2}(\Op)  \y, \ldots, f_{\alpha_n}(\Op)  \y \}.  \]
The acceleration method of \cite{ChenPere15}  computes 
\begin{equation}\label{eq:xrs} x_{\RS,n}:= \argmin_{x \in \RS^{n}} \|  A x - y\|.  \end{equation}
In contrast to the case for $\KS$ and the CGNE method, the solutions here cannot be calculated recursively, 
but the full Gramian matrix $G$ has to be set up and inverted \cite{Pereveryzevbook}: Formally,  the method 
computes 
\begin{equation}\label{aggmeth} 
\begin{aligned}
G &:= \skp{ A x_{\alpha_i}, A x_{\alpha_j}}_{i,j = 1,N} &\quad g &= \skp{y,A x_{\alpha_i}}_{i=1, N}, \\
 c&= G^{-1} g,& \quad  x_{\RS,n} &= \sum_{i=1}^N  c_i  x_{\alpha_i}.  
 \end{aligned} \end{equation} 
Here, neither the invertibility of the Gramian $G$ nor its well-conditioning is automatically guaranteed but has
to be monitored and algorithmically controlled. Although, the computational complexity of the method 
is mainly dominated by the calculation of the Tikhonov-regularized solutions $x_{\alpha_i}$, 
compared to the CGNE method, it has the drawback that whenever we increase the Krylov space 
by adding new elements $x_{\alpha_i}$, the full Gramian matrix has to be recalculated and inverted. 
It does not seem possible to design an equivalent iterative method that computes $x_{\RS,n}$
from previous solutions. 

The main contribution of this article is to show that such an iterative  algorithm with a short-term recursion 
for the acceleration method is, however, possible if we modify the rational Krylov space $\RS^{n}$ slightly.   
For this we  employ a fruitful theorem of Prani\'{c} and Reichel 
\cite{prarei,prareirod}, according to which  short  recursions are obtained by alternately mixing the Krylov spaces $\RS^{n}$ and 
$\KS^{n}$ . 
Thus, we define the following mixed rational Krylov space: 
\begin{equation} \begin{split} 
 \KRS^{2k}&:= \spa\big\{\y, f_{\alpha_1}(\Op) \y, \Op y,  f_{\alpha_2}(\Op), 
  \Op^2 y, \ldots ,\\
  & \qquad \qquad f_{\alpha_k}(\Op) \y \big\}  \\
   \KRS^{2k+1}&:= \spa\big\{\y, f_{\alpha_1}(\Op) \y, \Op y,  f_{\alpha_2}(\Op), 
  \Op^2 y, \ldots , \\ 
   & \qquad \qquad   f_{\alpha_k}(\Op) \y, \Op^k y  \big \}. 
\end{split} 
\end{equation}
Hence, 
\[ \KRS^n = \begin{cases} \RS^k \cup \KS^k &  n = 2 k, \\ 
             \RS^k \cup \KS^{k+1} &  n = 2 k +1. \\ 
            \end{cases}
\]
Therefore, the spaces $\KRS^n$ are build by adding alternately an element of the usual Krylov space $\KS$ and 
the rational Krylov space $\RS$. In the following we refer to the even iteration numbers $n = 2k$ as
{\em rational steps} (where a term $(\Op +\alpha_k I)^{-1}\y$ is added) and to the odd iterations 
$n = 2k +1$ as {\em Krylov steps} (where a term $\Op^k \y$ is added).

In analogy to the above, we define 
\begin{equation}\label{eq:xkrs} x_{\KRS,n}:= \argmin_{x \in \KRS^{n}} \|  A x - y\|.  \end{equation}
Algorithm~\ref{alg:ratcg} below provides a recursive method, called rational CG, that computes 
these minimizers recursively. 
 
\subsection{Rational representation}
Before we develop the method, we study a representation of the Krylov spaces by rational functions: 
Denote by $\PS^n$ the space of all polynomials of degree less than $n$. 
 (We denote by 
$ \lfloor . \rfloor$ the floor function, i.e., the rounding to the next smaller integer). 
\begin{proposition}
The Krylov spaces defined above have the following representation: 
\begin{alignat}{2}  
\KS^n &= p_{n-1}(\Op) \y, & \quad &p_{n-1} \in \PS^{n-1},  \nonumber \\
\RS^n & = r_{n-1} (\Op) \y, & &r_n(x) =  \frac{p_{n-1}(x)}{\Pi_{i=1}^{n} (x + \alpha_i) },  \qquad 
p_{n-1} \in \PS^{n-1},  \label{rsnum} \\
\KRS^n & = s_{n-1}(\Op) \y,   &
&s_{n-1}(x) = \frac{p_{n-1}(x)}{\Pi_{i=1}^{k} (x + \alpha_i) },  \qquad 
p_{n-1} \in \PS^{n-1},   k = \lfloor \frac{n}{2} \rfloor\, .  \label{zzz}
\end{alignat}
\end{proposition}
\begin{proof}
The case for $\KS^n$ is obvious as is the case $\RS^1$. By induction, let the statement be true for $\RS^{n-1}$. 
Then, by construction,  $\RS^n$ has a  representation  $\RS^n =  \tilde{r}_{n} (\Op) \y$ with 
\begin{equation}\label{eq:rep1} \tilde{r}(x)  =  \frac{p_{n-2}(x)}{\Pi_{i=1}^{n-1} (x + \alpha_i) }  + \frac{c_n}{x +\alpha_n} = 
 \frac{p_{n-2}(x) (x +\alpha_{n}) 
+ c_{n}\Pi_{i=1}^{n-1} (x + \alpha_i) }{\Pi_{i=1}^{n} (x + \alpha_i) }
\end{equation}
and $p_{n-2} \in \PS^{n-1}$.
The denominator represents an element in $\PS^{n-1}$, and, given an arbitrary polynomial $p_{n-1}(x)$ in $\PS^{n-1}$, 
take $ c_n =-\frac{1}{\Pi_{i=1}^{n-1} (-\alpha_{n} + \alpha_i) } p(-\alpha_{n})$. Then 
$ p_{n-1}(x) - c_{n} \Pi_{i=1}^{n-1} (x + \alpha_i)$  has a root at $-\alpha_n$ and thus can be factorized 
as $(x+ \alpha_n) p_{n-2}(x)$ with $p_{n-2} \in  \PS^{n-2}$. This allows one to represent $p_{n-1}(x)$ in 
the form of the numerator of \eqref{eq:rep1}. Thus the representation of $\RS^n$ is shown. 
From these results, it follows that the rational function $s_{n-1}(x)$ for  $\KRS^n$  is given by 
\begin{align*} s_{n-1}(x) &=  p_{m-1}(x) +   \frac{q_{k-1}(x)}{\Pi_{i=1}^{k} (x + \alpha_i) }   \\
& = 
\frac{1}{{\Pi_{i=1}^{k} (x + \alpha_i) }} \left(  p_{m-1}(x) {\Pi_{i=1}^{k} (x + \alpha_i) } + q_{k-1}(x) \right),  \end{align*} 
where   $m = k$  if  $n = 2k$  and $m = k +1$  if  $n = 2k +1$.  
Thus, in any case the numerator is an element in $\PS^{n-1}$.   Conversely, given a polynomial $p_{n-1}$, a 
factorization by $\Pi_{i=1}^{k} (x + \alpha_i) $ yields 
\[ p_{n-1}(x) = \Pi_{i=1}^{k} (x + \alpha_i)  q(x) + r(x), \] 
where $r$ has degree at most $k-1$ and $q$ has degree at most $n-1 -k = m-1$. Dividing by $\Pi_{i=1}^{k} (x + \alpha_i) $,
the representation of  $s_{n-1}$ via  $\frac{p_{n-1}(x)}{\Pi_{i=1}^{k} (x + \alpha_i) }$ can be written as a sum of elements 
of the Krylov space and the rational Krylov space. 
\end{proof}

In a similar way, we may define the residual spaces of the normal equations 
(note that this means $A^*A x  - A^*y$) 
for the mentioned method. 
For each of the above Krylov spaces, we define 
\begin{equation}\label{eq:resspace} \Q_{\mathcal{X}} := \spa\{\Op x-\y| x \in \mathcal{X}  \} 
 \qquad \mathcal{X} \in \{ \KS^n, \RS^n, \KRS^n \}. 
\end{equation} 
Let  $\PS_{1}^n$ be  the space of polynomials $p_n(x)$ of degree at most $n$ that satisfy 
\[ p_n(0) = 1. \] 
\begin{proposition}\label{prop2}
Let $\Q_{\KS^n}$. $\Q_{\RS^n}$, $\Q_{\KRS^n}$ be the residual spaces in \eqref{eq:resspace}. Then 
\begin{align}  
\Q_{\KS^n}  &= p_{n}(\Op) \y, \qquad p_{n} \in \PS_1^{n},  \nonumber \\
\Q_{\RS^n} & = r_{n} (\Op) \y, \qquad r_n(x) =  \frac{p_{n}(x)}{\Pi_{i=1}^{n} (\frac{x}{\alpha_i} + 1) },  \qquad 
p_{n} \in \PS_1^{n}, \label{resrsnum}\\
\Q_{\KRS^n} & =  s_n(\Op) \y, \quad 
s_n(x) = \frac{p_n}{\Pi_{i=1}^{k} (\frac{x}{\alpha_i} + 1) }, \qquad  p_{n} \in \PS_1^{n}, 
k = \lfloor \frac{n}{2} \rfloor\, . \label{reskrsnum}
\end{align}
\end{proposition}
\begin{proof} The case for $\Q_{\KS^n}$ is obvious as is the case $\Q_{\RS^1}$. The residual space $\Q_{\RS^n}$
consists of functions $\tilde{r}_n(\Op)\y$ with $\tilde{r}_n = x r_{n-1}(x) -1$ and $r_{n-1}$ as in 
\eqref{rsnum}. By dividing with $\sigma:= \Pi_{i=1}^n \alpha_i$, $\tilde{r}_n$ can be written as in \eqref{resrsnum}
with numerator  
\begin{equation}\label{blabel} \tilde{p}_n(x) =\sigma^{-1} p_{n-1}(x) x - \Pi_{i=1}^{n} (\tfrac{x}{\alpha_i} + 1). \end{equation} Clearly this is an 
element in $\PS_1^n$. Conversely given any polynomial $\tilde{p}_n(x) \in \PS_1^n$ it follows that 
$\Pi_{i=1}^{n} (\frac{x}{\alpha_i} + 1) -p_n(x) $ has a root at $0$ and dividing by $x$ yields a polynomial 
$p_{n-1}(x)$  as in \eqref{blabel}. This verifies the representation.  The result for $\Q_{\RS^n}$ is obtained 
in a similar way by replacing $\Pi_{i=1}^{n} (\tfrac{x}{\alpha_i} + 1)$ by 
$\Pi_{i=1}^{k} (\tfrac{x}{\alpha_i} + 1)$. 
\end{proof}
A consequence is the following useful result:
Let $n = 2k$, i.e., a rational step. Then 
\be\label{resrep}
(\Op + \alpha_k)^{-1} \Q_{\KRS^{n-1}} \subset \KRS^n, \qquad 
(\Op + \alpha_k)^{-1} \KRS^{n-1}  \subset \KRS^n. \\
\ee
For a Krylov step $n = 2k+1$ we have 
\be\label{resrepkry} 
\Op \Q_{\KRS^{n-1}} \subset \KRS^n,  \qquad 
\Op  \KRS^{n-1}  \subset \KRS^n. \\
\ee

Note that the same representation holds for the least-squares residual $A x- y$ but with 
$\Op$ replaced by $\Ap$ and $\y$ replaced by $y$. This follows since 
\[ A f(\Op)\y - y = (\Ap f(\Ap) - 1) y. \] 
As a further consequence, we can estimate the residual for the rational various Krylov methods by 
that of the standard Krylov methods: 
\begin{theorem}
Let $ x_{\KS,n}$, $x_{\RS,n}$, and $x_{\KRS,n}$ be defined as in 
\eqref{eq:xks}, \eqref{eq:xrs}, and \eqref{eq:xkrs}, respectively. 
Then, 
\begin{align*} \|A  x_{\RS,n} -y\| &\leq
\|A  x_{\KRS,n} -y\| \leq 
\|A  x_{\KS,n} -y\|,  
\end{align*} 
and 
%\frac{1}{\Pi_{i=1}^n \left(\frac{\|\Ap\|}{\alpha_i} +1 \right)}\|   \\
\begin{align*} \|A  x_{\KS,n} -y\| &\leq 
\Pi_{i=1}^k \left(\tfrac{\|\Ap\|}{\alpha_i} +1 \right) \|A  x_{\KRS,n} -y\|  \\
& \leq 
\Pi_{i=1}^n \left(\tfrac{\|\Ap\|}{\alpha_i} +1 \right)  \|A  x_{\RS,n} -y\|. 
%\end{align*}
%\begin{align*}
%\frac{1}{\Pi_{i=1}^n \left(\frac{\|\Ap\|}{\alpha_i} +1 \right)}\|A  x_{\KS,n} -y\|  &\leq \|A  x_{\KRS,n} -y\| \leq \|A  x_{\KS,n} -y\| 
\end{align*}
\end{theorem}
\begin{proof}
We have  with the notation in Proposition~\ref{prop2},  and by the definition of  $ x_{\KS,n}$, 
\begin{align*} \|A  x_{\RS,n} -y\|  &= \inf_{p_n \in \PS_1^n} \| \Pi_{i=1}^{n} (\tfrac{\Ap}{\alpha_i} + I)^{-1}  p_n(\Ap) \ypure \| \\
 &\leq  \Pi_{i=1}^{n} \|(\tfrac{\Ap}{\alpha_i} + I)^{-1} \| \inf_{p_n \in \PS_1^n}  \|p_n(\Ap)\ypure \| \\
 &\leq \inf_{p_n \in \PS_1^n}  \|p_n(\Ap) \ypure \| = \|A  x_{\KS,n} -y\|. 
\end{align*} The result for $x_{\KRS,n}$ is obtained in a similar way and so are the opposite directions of the estimates. 
\end{proof}

As an immediate consequence, it follows that $n$-dimensional linear systems,
$A^*A \in \R^{n\times n}$, 
can be solved 
by rational Krylov method in at most $n$ steps, assuming exact arithmetic. 
In fact, since  $\|A  x_{\KS,n} -y\|$ vanishes after at most $n$ steps, we obtain:
\begin{corollary}
Let $A \in \R^{n\times n}$, $y \in \R^n$ be such that $A x = y$ has a unique solution $x \in \R^n$. 
Then in at most $n$ steps we have $x_{\RS,n} = x$ and $x_{\KRS,n} = x$.  
\end{corollary}
\begin{remark}
In a similar way we can prove that the least-squares residual for the rational methods $x_{\RS,n}$ 
$x_{\KRS,n}$ are always smaller than that any of the appearing Tikhonov regularization $x_{\alpha_i}$ and 
the corresponding iterative Tikhonov regularization (cf.~\cite{HaGr}). 
\end{remark}

\section{An Arnoldi relation}\label{sec:arnoldi}
As mentioned above, the pure rational Krylov spaces $\RS^n$ do not allow for a simple iterative computation of 
the minimizers $x_{\RS,n}$. This is in contrast to the mixed space $\KRS$, as we will show below. 
The central observation to obtain this result is a theorem of Prani\'c and Reichel \cite[Theorem 1, 2]{prarei},  according 
to which the orthonormalized residuals in $\KRS$ allow for a short recurrence relation. In fact, the results 
\cite{prarei} deal with more general rational functions and also general ``mixing'' method; the main point 
is that the length of the recurrence is dominated by the length between polynomial powers $\Op^k$ occurring. 
Since in our definition we have an interlacing of such powers,  we can find simple recursive methods for 
computing minimizers in $\KRS$. Let us note that the Prani\'c-Reichel theorem has been employed in 
\cite{prareirod} for finding a Hessenberg-type representation with few off-diagonals via the Arnoldi process 
of a (not necessarily symmetric) matrix. In the first part we reproduce these results for the (simpler) symmetric case,  
and we obtain a similar Arnoldi relation as in \cite{prareirod}, which in our case is  pentadiagonal by symmetry. 
In the next section we illustrate the associated Lanczos method. 
As for the Krylov methods, we have to take care of the (rare) occasion of a breakdown: 
\begin{definition}\label{def:breakdown}
Let $\X$ be  any of the Krylov spaces $\X \in \{\KS^n, \RS^n,\KRS^n\}$.  We say that the respective space $\X$ 
does not break 
down at step $n$ if 
\[ \dim \X = n. \]
\end{definition}
The criterion for breakdown is well-known in the Krylov space case and can be extend to the rational cases: 
\begin{proposition}\label{propbreak} 
Let $n_{bd}$ be the smallest iteration number where one of  the Krylov spaces $\X \in \{\KS^n, \RS^n,\KRS^n\}$
breaks down, i.e., $n_{bd}-1 = \dim \X < n_{bd}$.
Then $n = n_{bd}$ if and only if $\y$ can be written as a linear combination of $n-1$ eigenvectors of $\Op$. 
\end{proposition}
\begin{proof}
If a breakdown occurs, then there must be a linear dependent combination of elements in $\X$ and hence there exist a 
polynomial $p_{n-1}$ of degree $\PS^n$ such that  
\[  \Pi_{i=0}^m (\Op + \alpha_k I)^{-1}  p_{n-1}(\Op) \y  = 0, \]
where $m = 0$ for $\KS$, $m = \lfloor \frac{n}{2} \rfloor$ for $\KRS$ and $m = n$ for $\RS$. Since 
the operator $\Pi_{i=0}^m (\Op + \alpha_k I)$ does not have a nullspace, in any case we end up with the condition 
$ p_{n-1}(\Op) \y = 0$. By \cite[Prop.~6-1, 6.2]{saadbook} this means that $\KS^n$ is an $n-1$-dimensional invariant 
subspace, thus any element can be written as a linear combination of $n-1$ eigenvectors, in particular $\y$ can be. 
Conversely if $\y$ can be expressed in this way, we may easily find a polynomial, e.g., $p_{n-1}(\lambda) = \Pi_{i=1}^n (\lambda-\lambda_i)$, 
where $\lambda_i$ are the eigenvalues of the corresponding eigenvectors. 
\end{proof}

\subsection{Arnoldi method and Arnoldi relation}
We now construct an orthonormal basis for $\KRS^n$ by orthonormalizing the basis elements in 
the definition by a (modified) Gram-Schmidt method. For the usual case $\KS^n$ this is exactly the 
Arnoldi method. We will show that in the symmetric case this yields a pentadiagonal representation of 
the operator $\Op$. We discuss Algorithm~\ref{alg:Arnolid}. 

\begin{algorithm}
\caption{Arnoldi method for $\KRS$.}\label{alg:Arnolid}
\begin{algorithmic}[1] 
 \STATE $q_1:= \frac{\y}{\|y\|}$ 
 \FOR{$i = 1\ldots N$}
     \STATE \[ v_i = \begin{cases} (\Op + \alpha_k I)^{-1} q_{i-1} & n = 2k \\ 
                   \Op q_{i-1} & n = 2k +1 \\ 
                 \end{cases} \] 
    \STATE    $ \tilde{q}_i := v_i - \sum_{j=1}^{j-1}  \skp{v_i,q_j} q_j $  \hfill  \COMMENT{Gram-Schmidt step} 
    \STATE    $ q_i = \frac{\tilde{q}_i}{\|\tilde{q}_i\|}$   \hfill \COMMENT{normalization step} 
  \ENDFOR
\end{algorithmic}
\end{algorithm}
Let us stress that in this method, the next element in the Krylov space $\KRS$ is given by either 
$(\Op + \alpha_k I)^{-1} q_{i-1}$ or $\Op q_{i-1}$ and not as suggested by the definition by 
$(\Op + \alpha_k I)^{-1} \y$ or $\Op^k \y$. 
This is motivated by the analogous construction for the 
usual Krylov space $\KS$ and an analogous suggestion in \cite[(4.13)]{Ruhe}. We have written the algorithm with an usual Gram-Schmidt orthgonalization but 
the modified Gram-Schmidt method is suited as well (and recommended) 
since  they are mathematically equivalent but the latter is numerically more stable.

Furthermore note that the algorithm is {\em well-defined} as long as $\|\tilde{q}_i\| \not = 0$ such that the normalization step can be performed.
We have the following proposition:
\begin{proposition}\label{prop4}
The iterations in Algorithm~\ref{alg:Arnolid} are well defined as long as the Krylov space $\KRS$ does not break down 
up to  dimension $n \leq N$. Moreover, in this case, the vectors $(q_i)_{i=1}^n$ build an orthonormal basis for $\KRS^n$. 
\end{proposition}
\begin{proof}
By \eqref{resrep}, \eqref{resrepkry}, the vectors $v_i$ are in $\KRS^{i}$ if the $(q_j)_{j=1}^{i-1} \in \KRS^{i-1}$. 
If $v_i \not \in \KRS^{i-1}$, then $\tilde{q}_i$ will be nonzero and the normalization step can be performed, and 
the iteration is well defined at $i=n$. If, however, $v_i \in \KRS^{i-1}$, that is, $\tilde{q}_i = 0$, 
we may find a rational function as \eqref{zzz} such that $s_{i-1}(\Op) \y = 0$, and hence as in the proof of
Proposition~\ref{propbreak} a polynomial $p_{i-1}$ with  $p_{i-1}(\Op) \y =0$, which contradicts 
the assumption of non-breakdown. Hence, by induction, the $(q_j)_{j=1}^i$ span $\KRS^{i}$ and they 
are orthonormal by construction. 
\end{proof}

%\todo{Proof that the $q_i$ span the space $\KRS$}

\begin{theorem}\label{th2}
Assume that the Arnoldi method in Algorithm~\ref{alg:Arnolid} 
does not break down up to index $N$. 
Then the matrix 
\[ T_{i,j}:= \skp{q_i, \Op q_j}_{i,j = 1,N} \]
is pentadiagonal and satisfies 
\begin{align*}  
T_{m,2k} = 0 \quad \text{ for }m = 2k+2,\ldots, N, \\
T_{m,2k+1} = 0 \quad \text{ for } m = 2k+3,\ldots, N. 
\end{align*} 
\end{theorem} 
\begin{proof}
Let $n$ be a Krylov step, i.e., $n = 2k +1$. By definition, we have 
\[ \|\tilde{q}_n\| q_n = \Op q_{n-1} - \sum_{j=1}^{n-1}  \skp{\Op q_{n-1},q_j} q_j. \]
Taking the inner product with $q_m$, first with $m >n$ and then with $m =n$  yields, by the orthogonality 
on the left-hand side, $\delta_{n,m}  \|\tilde{q}_n\|$. 
Thus, 
\[ \skp{q_{m},\Op q_{n-1}} = 0, \quad m \geq n+1, \quad \text{ and } \quad \skp{q_{n},\Op q_{n-1}} = \|\tilde{q}_n\|. \]
In case of a rational step, i.e., $n = 2k$, we multiply the equation for $v_i$ by 
$\Op + \alpha_k I$ to get 
\begin{equation}\label{aaa} \|\tilde{q}_n\| (\Op q_n + \alpha_k  q_n )   = 
q_{n-1} - \sum_{j=1}^{n-1} \skp{ (\Op + \alpha_k I)^{-1} q_{n-1} ,q_j} (\Op + \alpha_k I) q_j .
\end{equation} 
Now  take  the inner product with $q_m$ for $m >n+1$. We prove by induction that 
\[ \skp{q_{m},\Op q_{n}} = 0, \quad \text{ for } m \geq n+2, n = 2k. \]
Assume that it holds up to all 
previous rational steps for  $n = 2j $, $j = 1, k-1$. 
Let $m = n +2$. By \eqref{aaa} it follows from the orthogonality and the fact that 
$\skp{q_m, \Op q_j}$ on the right-hand side vanishes by the induction assumption for 
all rational steps and for the Krylov steps as well by the proof in the first step. 
Thus $T$ has at most two nonzero lower subdiagonals. Since $T$ is symmetric, the 
matrix is pentadiagonal with the stated structure. 
\end{proof} 

As an illustration, the matrix $T$ has the following form. Denote by $Q_n :\R^{n} \to \X$ the 
following operator that takes linear combinations of the vectors $q_i$ given by the Arnoldi method:
\be\label{defqn} Q_n = [q_1 \, \cdots \, q_n] .  \ee
Then  the resulting matrix $T$ has the following 
structure: 
\be\label{stru} T= Q_n^T \Op Q_n = \tmat{\grdiag}{\beta}{\gamma}. \ee
That is, a pentadiagonal matrix with zeros in the second diagonal at even indices. 
One may as well view this as a block-tridiagonal matrix
consisting of $2\times2$ blocks with  with rank-1 off-diagonal blocks. 

We note that a corresponding representation has been established in \cite{prareirod} 
for the non-symmetric case, where the matrix $T$ has generalized Hessenberg form with 
several subdiagonals; see, e.g., equation~(18) ibid. Clearly, in case of   symmetric operators, 
such Hessenberg structure becomes our multi-diagonal form.

%\todo{Note that Lothar and giuseppe has done this already for the nonsymmetric case} 

\section{A Lanczos method}\label{sec:lanczos}
Based on the Arnoldi relations, we can now solve the least-squares problem 
by exploiting the structure of $T$. This is analogous to the usual Krylov space methods, 
where a tridiagonal structure appears, which can be inverted iteratively yielding the 
Lanczos method.  Similar to, e.g.,  \cite[Chpt.~6.7]{saadbook}, we extend this idea to our pentadiagonal case. 

For simplicity of notation we assume now a finite-dimensional case with $\Op$ given by a
symmetric $N \times N$ matrix and $\y$ a given vector in $\R^N$. 
Assume that the Arnoldi method does not break down up to index $N$. 

We define the matrix (respectively operators) 
\[ T_n = Q_n^T \Op Q_n  = (T_{i,j})_{i,j = 1,m},  \]
with $Q_n$ from \eqref{defqn},  
and we set $Q = Q_N$. 
It follows that 
\[ Q^T \y = \beta e_1, \]
where $e_1$ is the first unit vector $e_1 = (1,0,\ldots, 0)^T$ and $\beta \in \R$.
By orthogonality, the normal equation $\Op x = \y$ translates to 
$T c =  \beta e_1$ where $ x = Q c$. Finding the least-squares minimizer $x_n$  in the Krylov space $\KRS^n$ 
translates to $Q_n^T \Op x_n = Q_n^T \y$ or 
\begin{equation}\label{defcn}   T_n c_n =  \beta e_1  \qquad \text{ with } \qquad  x_n = Q_n c_n.
\end{equation} Thus, by the 
sparse structure of $T$, we can iteratively solve for $x_n$, which is the objective of the next 
section. The result  is essentially a generalization of the D-Lanczos method in \cite[Section 6.7.1]{saadbook}, 
and we follow a similar analysis as there.

\subsection{Solving the pentadiagonal system}
Recall that we have shown the following structure:
In case of a rational step $n = 2k$:
\begin{equation}\label{strrat} T_n = 
\begin{bmatrix} T_{n-1} &  \begin{matrix} 0 \\ \beta_n \end{matrix} 
   \\
\begin{matrix} 0 & \beta_n    \end{matrix} &
%\end{matrix}  
   \grdiag_{n} 
\end{bmatrix}.
\end{equation}
In case of a Krylov step $n = 2k +1$ we have 
\begin{equation}\label{strkry} T_n = 
\begin{bmatrix} T_{n-1} &  \begin{matrix} 0 \\ \gamma_n \\ \beta_n \end{matrix} 
   \\
\begin{matrix} 0 & \gamma_n & \beta_n    \end{matrix} &
%\end{matrix}  
   \grdiag_{n} 
\end{bmatrix} = 
\begin{bmatrix}\begin{matrix} T_{n-2} &  \begin{matrix} 0 \\ \beta_{n-1} \end{matrix} 
   \\
\begin{matrix} 0 & \beta_{n-1}    \end{matrix} &
%\end{matrix}  
   \grdiag_{n-1} 
\end{matrix}  &  \begin{matrix} 0 \\ \gamma_n \\ \beta_n \end{matrix} 
   \\
\begin{matrix} \hspace*{-4mm}0 &  & \hspace*{-1mm}\gamma_n & \hspace*{1mm}\beta_n    \end{matrix} &
   \grdiag_{n} 
\end{bmatrix}  
\end{equation}
%
%
%Assume that at step $n\geq 2$ we have defined $x_{n-1}$ and $c_{n-1}$ by $T_{n-1} c_{n-1} = \beta e_1 $.
Define the vector $d_n \in \R^{n}$ by  
\begin{equation}\label{defdn} T_n d_n := e_n,  \end{equation}
where $e_n \in \R^n$ is the $n$th unit vector $e_n = (0,\ldots,0,1)^T$. 
In the following, we denote the last entry of $d_n$ by $\tau_n$, i.e.,  
\be\label{deftau} d_n = \begin{bmatrix} * \\ \tau_n \end{bmatrix}. \ee 
We show that with some $\xi_n \in \R$, we have the recursion 
\be\label{plugin} c_n = \begin{bmatrix} c_{n-1} \\ 0 \end{bmatrix} + \xi_n d_n  .\ee

Indeed, at first assume that $n$ is a rational step, i.e., $n = 2k$. Then plugging in 
\eqref{plugin} into \eqref{defcn}, noting that $T_{n-1} c_{n-1} = \beta e_1$ and $T_n d_n = e_n$,  yields the condition 
\[  \begin{bmatrix} 0 \\ \beta_n c_{n-1;n-1} \end{bmatrix} + 
\xi_n  e_{n} = 0, \]
where we denote by $c_{n-1;k}$ the $k$-th entry of the vector $c_{n-1}$. 
This gives 
\[ \xi_n = -\beta_n   c_{n-1;n-1}  = -\beta_n c_{n-1}^T e_{n-1}. \]
Next, we derive a recursion for $d_n$.  Let $\sigma_n$ be a real number in the ansatz 
\be\label{defdnrec}  d_n = \sigma_n   \begin{bmatrix}  d_{n-1} \\  0 \end{bmatrix} + \tau_n  e_n, \qquad 
\text{i.e.,} \qquad  d_n =  \begin{bmatrix}  * \\ \sigma_n \tau_{n-1} \\  \tau_n \end{bmatrix}, 
\ee
with $\tau_n, \sigma_n \in \R$. The defining equation for $d_n$, \eqref{defdn}, 
yields the condition for $\tau_n, \sigma_n $  
(noting that $T_{n-1} d_{n-1} = e_{n-1}$ and \eqref{strrat}) 
\[  \begin{bmatrix}0 \\ \sigma_n   \\ \sigma_n \beta_n d_{n-1;n-1} \end{bmatrix} 
+ \tau_n \begin{bmatrix} 0 \\ \beta_n  \\ \grdiag_n \end{bmatrix}  = e_n,\]
 which can be resolved, noting that $ d_{n-1;n-1} = \tau_{n-1}$ by definition, to
\[ \sigma_n = - \tau_n \beta_n \qquad  \sigma_n \beta_n \tau_{n-1} + \tau_n \grdiag_n = 1. \]
 This yields the recursion 
 \[ \tau_n = \frac{1}{\grdiag_n - \beta_n^2 \tau_{n-1}}  \qquad \sigma_n =  - \tau_n \beta_n.\]
Thus, both $c_n$ and $d_n$ can be calculated by a 1-step recursion via \eqref{defdnrec} and \eqref{plugin}.
 
Next, consider the Krylov step $n = 2k +1$ (with $n>1$).  Again we show a recursion 
\eqref{plugin}. Plugging this identity into the equation $T_n c_n = \beta e_n$ and 
noting $T_n d_n = e_n$ yields the condition  
\[  \xi_n = -\gamma_n c_{n-1;n-2} - \beta_n c_{n-1;n-1} = 
-\gamma_n c_{n-1}^T e_{n-2}  - \beta_n c_{n-1}^T e_{n-1}  .
\]
The recursion for $d_n$ is now more involved as we need a 2-step recursion. 
Indeed, we make the ansatz 
 \be\label{defdnkry} d_n =  \sigma_n \begin{bmatrix}  d_{n-1} \\ 0 \end{bmatrix} + 
 \tau_n e_n + \eta_n  \begin{bmatrix}  d_{n-2} \\ 0 \\  0 \end{bmatrix},  
 \ee
 with parameter $\sigma_n, \eta_n,\tau_n$, 
and plug this into  \eqref{defdn}.
The last three rows yield a linear equation 
for the coefficients $\sigma_n,\tau_n,\eta_n$:
\[ \sigma_n  \begin{bmatrix} 0 \\  1  \\ \beta_n d_{n-1;n-1} + \gamma_n d_{n-1;n-2} \end{bmatrix} 
+ \tau_n \begin{bmatrix}  \gamma_n \\ \beta_n \\ \grdiag_n \end{bmatrix} 
+ \eta_n  \begin{bmatrix}  1 \\ d_{n-2;n-2}  \beta_{n-1} \\ d_{n-2;n-2}  \gamma_n \end{bmatrix} =   \begin{bmatrix} 0 \\ 0 \\ 1 
   \end{bmatrix} \, .\]
Note the identities 
\[ d_{n-1;n-1} = \tau_{n-1} ,\qquad 
d_{n-1;n-2} = \sigma_{n-1} \tau_{n-2},  \qquad 
d_{n-2;n-2} = \tau_{n-2}. \]
The recursion for $c_n$ and $d_n$ is now given by 
\eqref{defdnkry} and \eqref{plugin}.
%which gives 

Let us now rewrite this recursion in terms of $x_n$. 
Define $x_n = Q_n c_n$, and $p_n = Q_n d_n$. 
Then \eqref{plugin} reads 
\be\label{xstep} x_n = x_{n-1} + \xi_n p_n, \ee
where 
\[ \xi_n = \begin{cases}  -\beta_n x_{n-1}^T q_{n-1} & n = 2k, \\
            -\gamma_n x_{n-1}^T q_{n-2}  - \beta_n x_{n-1}^T q_{n-1} &  n = 2k+1 ,  
           \end{cases} \]
whereas the recursion for $p_n$ is 
in the rational case $n = 2k$:  
\be\label{ratstep}  p_n = \sigma_n p_{n-1} +  \tau_n q_n .\ee
In the Krylov case $n = 2k +1$, we have 
\be\label{krystep}  p_n = \sigma_n p_{n-1} + \eta_n p_{n-2} + \tau_n q_n.  \ee
where $\xi_n, \sigma_n, \tau_n, \eta_n$ are calculated by the recursions above.
The recursion starts at $n = 1$ by a direct calculation of 
\[x_1 = \frac{\skp{\y,\y}}{\skp{\Op \y,\y}} \y ,\qquad 
q_1 = \frac{\y}{\|\y\|}, \qquad 
\tau_1 = \frac{\skp{\y,\y}}{\skp{\Op \y,\y}}, \qquad  p_1  = \tau_1 q_1. \]
We note that the rational step above for $p_n$ is valid for $n =2$. 
The full method for calculating least-squares solutions in $\KRS$ is now
presented in Algorithm~\ref{alg:lanz} as a pseudo-Matlab code. 
The routine $\text{GramSchmid}(v,Q)$ there means a orthgonalization {\em and} normalization step as 
in the calculation of $q_i$ via $\tilde{q}_i$ in the Arnoldi method in Algorithm~\ref{alg:Arnolid}.

\begin{algorithm}
\caption{Lanczos Algorithm for $\KRS$.}\label{alg:lanz}
\begin{algorithmic}[1]
 \STATE $\Op = A^*A$, $\y = A^*y$ 
 \STATE $\tau = \frac{\skp{\y,\y}}{\skp{\Op \y,\y}}$, 
 $Q(:,1) = \frac{\y}{\|\y\|}$, $p = \tau Q(:,1)$, 
 $x = \frac{\skp{\y,\y}}{\skp{\Op \y,\y}} \y$
 \FOR {$n = 2\ldots \MAXIT$}
  \IF[Rational step]{$n$ is even}
          \STATE $k := \frac{n}{2}$            
          \STATE $v =  (\Op +\alpha_k I)^{-1} Q(:,n-1)$ 
          \STATE $q = \text{GramSchmid}(v,Q)$ 
          \STATE $Q(:,n) = q$ 
          \STATE $\grdiag = q^T \Op q$, $\beta = Q(:,n-1)^T \Op q$ 
          \STATE $\tau_{old} = \tau$ 
          \STATE $\tau = 1/(\grdiag - \tau \beta^2)$ 
           \STATE $ \sigma = -\tau \beta$ 
           \STATE $p_{old} = p$ 
           \STATE $p = \sigma p + \tau q$ 
            \STATE $\xi = -\beta x^T Q(:,n-1)$ 
            \STATE $x = x + \xi p$ 
      \ELSE[Krylov step]   
           \STATE $v =  \Op  Q(:,n-1)$ 
          \STATE $q = \text{GramSchmid}(v,Q)$ 
          \STATE $Q(:,n) = q$ 
          \STATE $\beta_{old} = \beta$ 
          \STATE $\grdiag = q^T \Op q$, $\ $  $\beta = Q(:,n-1)^T \Op q$, $\ $  $\gamma=  Q(:,n-2)^T \Op q$ 
          \STATE $M = \begin{bmatrix} 0 & \gamma & 1\\ 1 & \beta & \tau_{old} \beta_{old} \\ 
          \beta\tau + \gamma \sigma \tau_{old} & \grdiag & \tau_{old} \gamma \end{bmatrix} $         
          \STATE $\tau_{old} = \tau$ 
          \STATE $\begin{bmatrix} \sigma \\ \tau \\ \eta \end{bmatrix} = M^{-1}  \begin{bmatrix} 0 \\ 0 \\ 1 \end{bmatrix}$
          \STATE $p_{new} = \sigma p + \eta p_{old} + \tau Q(:,n)$
          \STATE $p_{old} = p$ 
           \STATE $p = p_{new} $ 
          \STATE $\xi = -\gamma x^T Q(:,n-2) - \beta x^T Q(:,n-1)$  
          \STATE $ x = x + \xi p$ 
       \ENDIF                                   
  \ENDFOR
\end{algorithmic}
\end{algorithm}

This algorithm is the implementation of the formulas defined above, and it is well-defined as long as the 
Krylov space $\KRS$ does not break down. The involved divisions, e.g., 
$\tau = 1/(\grdiag - \tau \beta^2)$ or matrix inversion (e.g., $M^{-1}$) are well-defined in case of non-breakdown 
in exact arithmetic because $T_n$ is then always invertible, and the formulas are derived from solving 
\eqref{defdn}.

%\todo{Write down the algorithm}

As in the standard case for $\KS$, 
the iterations of Algorithm~\ref{alg:lanz} satisfy certain 
orthogonality relations: 
\begin{proposition}
Define $p_n := p$ in Algorithm~\ref{alg:lanz} at iteration $n$. Then 
\[ \skp{\Op p_n, p_k} = 0 \qquad \text{ for } n \not = k. \]
\end{proposition}
\begin{proof}
Note that by construction $p_n = Q_n T_n^{-1} e_n$. 
Since $T_n = Q_n^T \Op Q_n$, we have that $Q_n^T \Op Q_n d_n = e_n$ and 
hence $Q_n^T \Op p_n = e_n$. As a consequence, $\Op p_n =  q_n + \spa\{q_j, j>n\}$. 
Thus, for $k< n$, 
\[ \skp{\Op p_n, p_k}  = \skp{q_n+ \spa\{q_j, j>n\}, p_k} = 0 \]
since $p_k \in \spa\{q_1,\ldots, q_k\}$ and the orthogonality of $q_i$.
By symmetry, the results holds also for $k>n$. 
\end{proof} 
This means that the matrix $\skp{\Op p_i, p_j}$, $i,j = 1\ldots,N$, is diagonal. 
\begin{proposition}
Consider the residual for the normal equations 
\[ r_n = \Op x_n - \y \]
for the iterations $x_n = x$ in Algorithm~\ref{alg:lanz} at iteration $n$.
The Gramian matrix for the residual vectors has the 
following nonzero entries:
\[ R_{i,j} := \skp{r_i,r_j} 
= \begin{bmatrix} * & * & & & & \\ 
 *  & * &  & & &\\
&  & * & * & &\\ 
& &  * &* &  &\\ 
& & &  & * & * \\
& & & & *  &  *
  \end{bmatrix}, 
\]
i.e., it is  tridiagonal with additional zeros in the lower off-diagonals for even  and 
in the upper off-diagonal for odd indices. 
Moreover, we have 
\be\label{ort} \skp{r_n,p_i} = 0, \quad \text{ for all }  i=1,\ldots, n. \ee
\end{proposition}
\begin{proof}
Let $\tilde{c}_n \in \R^N$ be a vector with values $c_n$ at position $1,\ldots,n$ and $0$ at the rest.
By the structure \eqref{stru},  \eqref{strkry}, \eqref{strrat}  
it follows for a rational step that 
\be\label{resrat} \tilde{r}_n := T \tilde{c}_n - \beta e_1 = \beta_{n+1} e_{n+1} \qquad \text{ for } n = 2k \ee
while in a Krylov step
\be\label{reskry} \tilde{r}_n:= T \tilde{c}_n - \beta e_1 = \beta_{n+1} e_{n+1} + \gamma_{n+2} e_{n+1}  \qquad \text{ for } n = 2k+1.  \ee
Thus in case of $n = 2k$, $\tilde{r}_n^T\tilde{r}_{n+1} = 0$, which verifies the $0$ in the first offdiagonals. 
In any case we have  $\tilde{r}_n^T\tilde{r}_{n+k} = 0$ for $k\geq 2$. The claimed matrix structure is now verified by 
the observation that $r_n = Q \tilde{r}_n$ and the orthogonality of $Q$. 

Identity \eqref{ort} follows  since $d_i$ is a vector of length $i$, hence $p_i \in \KRS^i$, while 
$\tilde{r}_n$ has zero entries in the first $n$ components. Thus $r_n$ is orthogonal to  $\KRS^n$ and hence to all $p_i$.
\end{proof} 

%\todo{Verify orthogonality relastoins}

\section{The rational CG method}\label{sec:ratcg}
The disadvantage of the previous Lanczos method in Algorithm~\ref{alg:lanz}  
is that the orthogonal vectors $q_i$ 
has to be saved, and thus the memory requirement might get large. 
 In this section we state an equivalent algorithm that avoids saving the $q_i$, and it is 
 directly based on one- and two-step recurrence relations for the vectors $p_n$ and $x_n$, 
 which have the same meaning as in the previous section. This is analogous to the derivation 
 of the CG method from the  D-Lanczos method from the orthogonality relations; cf.~\cite[Section~6.7]{saadbook}. 
 
We consider the iteration \eqref{xstep}, \eqref{ratstep}, and \eqref{krystep}:
First consider \eqref{xstep}: Assume that $p_n$ has been computed, and denote by $r_n$ the residual $\Op x_n -\y$. 
Then 
\[ x_n = x_{n-1} + \xi_n p_n  \Rightarrow r_n = r_{n-1} + \xi_n \Op p_n. \]
We require that $r_n$ is orthogonal to $p_n$, hence 
\be\label{xin} 
\xi_n = -\frac{\skp{r_{n-1},p_{n}}}{\skp{\Op p_n, p_n}} .
\ee
Next we consider the iterations for $p_n$. 
Rewriting it without usage of the $q_i$ is not so difficult for a Krylov step:
Let $n = 2k +1$ and consider \eqref{krystep}. It follows from \eqref{resrat} (after multiplying with $Q$) 
that  $r_{n-1} \sim q_n$. Thus, the recursion   \eqref{krystep} can be replaced by 
\be\label{krystepmod}  p_n = \sigma_n p_{n-1} + \eta_n p_{n-2} + \tau_n' r_{n-1},  \ee
where $r_{n-1}$ is the  residual of the previous step 
and $\tau_n'$ is some constant (possibly different from $\tau_n$). 
To fix the constants, we observe from \eqref{xin}  that  $p_n$ may be rescaled by a multiplicative constant without changing the method.  Thus we set $\tau_n' =1$. The other constants  are obtained by 
forcing  the orthogonality relations $\skp{\Op p_i, p_j } \sim \delta_{ij}$ to hold. This yields 
the iteration for a Krylov step, i.e., $n = 2k +1$:
\be\label{cfeq} 
p_n = -\frac{\skp{r_{n-1}, \Op p_{n-1}}}{\skp{\Op p_{n-1},p_{n-1}}} p_{n-1} 
-\frac{\skp{r_{n-1}, \Op p_{n-2}}}{\skp{\Op p_{n-2},p_{n-2}}} p_{n-2}  + r_{n-1} .
\ee

The recursion for a rational step $n = 2k$ is more involved. Considering \eqref{ratstep}, the problem 
is that  $q_n$ does not have a simple expression in terms of residuals. However, 
by \eqref{reskry}, \eqref{resrat},
it may be written as a linear combination of $r_{n}$ and $r_{n-1}$. From    \eqref{xstep} it follows that 
$r_n = r_{n-1} - \xi_n \Op p_n$. Thus $p_n$ can be written as linear combination of $p_{n-1}$, $r_{n-1}$, and 
an implicit term $\Op p_n$. Keeping in mind \eqref{resrep} and that $p_{n}$ should represent an element in 
$\KRS^n$ leads to idea that the factor in front of $\Op p_n$ should be an $-\alpha_k^{-1}$. Thus, 
the form of the rational step should be (again using on  rescaling degree of freedom for the factor in front of $r_{n-1}$) 
\[ p_n = (\Op +\alpha_k I)^{-1} \left[ \factor_n p_{n-1} + r_{n-1} \right]. \]
%which has also been verified numerically. 
It remains to fix the factor $\factor_n$, which is obtained through the orthogonality relations as before. 
Putting the operator $ (\Op +\alpha_k I)$  on the right-hand side yields 
\be\label{bab}  \skp{(\Op +\alpha_k I) p_n, p_{n-1}} = \factor_n \skp{p_{n-1},p_{n-1}}  + \skp{r_{n-1},p_{n-1}} .\ee
Thus, requiring $\skp{\Op p_n,p_{n-1}} = 0$  and  $\skp{r_{n-1},p_{n-1}} = 0$ yields 
\begin{align*}  \factor_n &= \frac{\alpha_k  \skp{p_n, p_{n-1}}} {\skp{p_{n-1},p_{n-1}} }  = 
 \frac{\alpha_k  \skp{(\Op +\alpha_k I)^{-1}[ \factor_n p_{n-1} + r_{n-1}]  , p_{n-1}}} {\skp{p_{n-1},p_{n-1}} },
\end{align*}
which gives 
\begin{align} \factor_n &=  \frac{\alpha_k \skp{(\Op +\alpha_k I)^{-1} p_{n-1} ,r_{n-1}}}
{\skp{p_{n-1},p_{n-1}} -\alpha_k \skp{(\Op +\alpha_k I)^{-1} p_{n-1} , p_{n-1}}}  \nonumber \\
&= 
\frac{\alpha_k \skp{ (\Op +\alpha_k I)^{-1} p_{n-1} ,r_{n-1}}}{\skp{\Op p_{n-1},  (\Op +\alpha_k I)^{-1} p_{n-1}}}  . \label{rhoform}
\end{align}
Another equivalent formula for $\factor_n$ is as follows: 
\begin{align} \factor_n &=  \frac{\skp{[\Op + \alpha_k - \Op ] p_{n-1} , (\Op +\alpha_k I)^{-1} r_{n-1}}}
{\skp{\Op p_{n-1},  (\Op +\alpha_k I)^{-1} p_{n-1}}} \nonumber  \\
& = 
 \frac{\skp{p_{n-1} , r_{n-1}} - \skp{\Op p_{n-1} , (\Op +\alpha_k I)^{-1} r_{n-1}}}
{\skp{\Op p_{n-1},  (\Op +\alpha_k I)^{-1} p_{n-1}}}  \nonumber \\
& = - \frac{ \skp{\Op p_{n-1} ,(\Op +\alpha_k I)^{-1} r_{n-1}}}{\skp{\Op p_{n-1},  (\Op +\alpha_k I)^{-1} p_{n-1}}}, \label{alteta}
\end{align}
exploiting the orthogonality $\skp{p_{n-1} , r_{n-1}}  = 0$. 

Thus, the formula for $p_n$ can be computed from \eqref{rhoform} by the steps 
\begin{align*} \pone_n&:= (\Op +\alpha_k I)^{-1} p_{n-1}, \\
 \ptwo_n&:= (\Op +\alpha_k I)^{-1} r_{n-1},  \\
 p_n &:= \frac{\alpha_k \skp{\pone_n, r_{n-1}}}{\skp{\Op p_{n-1}, \pone_n}} \pone_n + \ptwo_n.
\end{align*}

\begin{remark}
The computation of $p_n$ in a rational step thus requires solving for $\pone_n$ and $\ptwo_n$, i.e., 
two linear solves with Tikhonov matrix $(\Op +\alpha_k I)$, which is an overhead compared to the Lanczos method above that needs only one. However, we did not find a step that requires only one linear solve. Note, however, that the system matrix 
(Tikhonov matrix) is in both cases that same, only the right-hand sides differ. 
In Matlab, this can be computed  by the statement 
\be\label{compcg} [\pone_n \ \ptwo_n] = (\Op +\alpha_k I)\backslash[p_{n-1} \  r_{n-1}], \ee 
and we found that it does not require much more additional computation  time. Observe that for a direct solver, the main 
work is in the matrix decomposition of the system matrix, which has to be performed only once per step. Using 
this for two right-hand sides is then negligible overhead work. 

Interestingly, if one insists on using only one linear system solve with one right-hand side per step, then this can be achieved by a  formula that 
uses complex variables. Indeed,  $\factor_n$ and the formula for $p_n$ can be rewritten by \eqref{alteta} 
\begin{align*} p_n &=  
\frac{1}{\skp{\Op p_{n-1},  \pone_n}} 
\Big[ 
 -\skp{\Op p_{n-1} ,\ptwo_n}  \pone_n + 
 \skp{\Op p_{n-1},  \pone_n} \ptwo_n \Big] \\
& = \frac{1}{\skp{\Op p_{n-1},  \pone_n}} 
\mathcal{I} \left[ \skp{\Op  p_{n-1}, \overline{\pone_n  + i \ptwo_n} } (\pone_n  + i \ptwo_n)  \right]. 
\end{align*}
Here $\overline{\pone_n  + i \ptwo_n}$ denotes the convex conjugate,   $\mathcal{I}$ is the imaginary part,
and $i$ the imaginary unit. 
In this formula, the common factor  $\frac{1}{\skp{\Op p_{n-1},  \pone_n}}$ can be ignored since a 
multiplicative factor of 
$p_n$ does not change the iteration. Moreover, in this formula we only require to calculate 
$(\Op +\alpha_k I)^{-1} [p_{n-1} + i r_{n-1}]$, i.e., one linear solves with only one complex right-hand side.  
In the numerical calculations, however, we did not find much of a benefit of using this formula. 
\end{remark}

Finally, we are in the position to present the full rational CG algorithm. The only thing remaining open is the
initial value for the  $p$-variables, i.e., $p_1$. Setting $x_{-1} = 0$ and by \eqref{xstep}, gives $p_1 = x$, 
noting the scaling freedom for the $p$-variables. 
Together, we obtain Algorithm~\ref{alg:ratcg} for solving the normal equations $\Op x = \y$. 

\begin{algorithm}
\caption{Rational CG method for $\KRS$.}\label{alg:ratcg}
\begin{algorithmic}[1] 
\STATE $\Op = A^*A$, $\y = A^*y$ 
 \STATE  $x =  \frac{\skp{\y,\y}}{\skp{\Op \y,\y}} \y$. 
  \STATE  $r =  \Op x - \y$. 
  \STATE $p  = x$.                   \COMMENT{$x_1,p_1,r_1$ are defined} 
 \FOR{$n = 2\ldots \MAXIT$}
      \IF{$n$ is even}  % \COMMENT{Rational Step}  
          %\begin{align*} 
          \STATE $k := \frac{n}{2}$            
           \STATE $\pone:= (\Op +\alpha_k I)^{-1} p$ 
           \STATE  $\ptwo:= (\Op +\alpha_k I)^{-1} r$   \COMMENT{Implementation as in \eqref{compcg}}
            \STATE $\zeta := -\frac{\skp{\Op r, \pone}}{\skp{\Op p, \pone}} $
            \STATE   $\pold = p$           \COMMENT{$\pold = p_{n-1}$ }  
            \STATE  $p := \zeta\pone + \ptwo$    \COMMENT{$p = p_{n}$ }  
      \ELSE     % \COMMENT{Krylov Step}  
           \STATE  $p = -\frac{\skp{r,\Op p}}{\skp{\Op p,p}} p - \frac{\skp{r,\Op \pold}}{\skp{\Op \pold,\pold}} \pold + r$
 %          \STATE  $ \pold = p $      \COMMENT{$\pold = p_{n-1}$ }  
%           \STATE  $ p = \pnew $       \COMMENT{$p = p_{n}$ } 
       \ENDIF                                   
            \IF{$\Op p = 0$} 
                 \STATE Terminate algorithm \COMMENT{Failure by Breakdown}
                 \ELSE
           \STATE  $\eta := \frac{\skp{r,p}}{\skp{\Op p,p}}$ 
           \STATE  $ x = x - \eta p $               \COMMENT{Solution $x_n$ at step $n$}
           \STATE  $ r = r - \eta \Op p $           \COMMENT{Step $n$ finished } 
           \ENDIF
  \ENDFOR
\end{algorithmic}
\end{algorithm}

Let us comment about the well-definedness of the algorithm. As long as the failure criterion 
$\Op p_n = 0$ is not satisfied, the fractions in the calculations 
are all well-defined: This is obvious for the terms $\skp{\Op p_n,p_n}$, $\skp{\Op \pold,\pold}$. 
The denominator $\skp{\Op p_n, \pone_n}$ is well-defined under the non-failure condition since 
\[ \skp{\Op p, \pone_n} = \skp{\Op (\Op +\alpha_k I) \pone_n,\pone_n} = 
\|\Op \pone_n\|^2 + \alpha_k \skp{\Op \pone_n,\pone_n}, 
\]
which is $0$ if and only if $0 = \Op \pone_n =  (\Op +\alpha_k I)^{-1} \Op p$. Since $\Op +\alpha_k I$ has no nullspace, 
this can only happen if $\Op p_n = 0$, which is excluded by the non-failure condition.

Finally, we show that the rational CG method does what it is supposed to do. 
\begin{theorem}\label{main}
Denote by $x_n,p_n,r_n$ the respective variables $x,p,r$ at iteration $n$ in  
Algorithm~\ref{alg:ratcg}. Then, they satisfy the orthogonality conditions 
\be\label{verify} \skp{\Op p_{n} ,p_j} = 0, \quad j=1,\ldots, n-1, \qquad \text{ and } 
 \qquad \skp{r_{n}, p_j} = 0, \qquad j=1,\ldots, n.
\ee
As a result, 
Algorithm~\ref{alg:ratcg}
computes at iteration $n$ the solution to the least-squares problem in the 
mixed Krylov space $\KRS^n$ \eqref{eq:xkrs}.
\end{theorem}
\begin{proof}
With the notation in the theorem, 
at the initialization, we obviously have that $x_1,p_1 \in \KRS^1$,  and $r_1 \in \Q_{\KRS^1} \subset \KRS^2$. 
Suppose that $x_{n-1},p_{n-1} \in \KRS^{n-1}$ and $r_{n-1} \in  \Q_{\KRS^{n-1}} \subset \KRS^n$. In a Krylov step, it follows immediately that $p_n \in \KRS^n$. 
In a rational step we have by \eqref{resrep} that $p_n \in \KRS^n$. 
Thus, in any case  $x_n \in \KRS^n$ and $r_{n} \in  \Q_{\KRS^{n}} \subset \KRS^{n+1}$.
By induction we have proven that 
\[ x_n,p_n  \in \KRS^n \qquad r_n \in \Q_{\KRS^n}. \]
Next, we verify the orthogonality conditions 
\eqref{verify}
by induction. 
Assume that \eqref{verify} holds with $n$ replaced by $n-1$. 

In a Krylov step, $n = 2k+1$.  it follows by construction of the coefficients, 
cf.~\eqref{cfeq}, that 
\[ \skp{\Op p_n, p_{j}} = 0, \quad j= n-1, n-2. \]
For $j<n-2$ we find by the induction hypothesis that, with some coefficients $b_1,b_2$
from \eqref{cfeq}, 
\[ \skp{\Op p_n, p_j} = 
\skp{b_1 p_{n-1} + b_2 p_{n-1} + r_{n-1} , \Op p_j} = 
\skp{r_{n-1},\Op p_j}. \]
However, $\Op p_j \in \KRS^{n-1}$, thus the right-hand side vanishes 
by induction hypothesis. Thus $\skp{\Op p_{n} ,p_j} \sim \delta_{n,j}$ is shown. 
The orthogonality of the residuals and $p_j$ in \eqref{verify} follows by $r_n = r_{n-1} -\eta \Op p_n$ 
and by construction: Taking the inner product with $p_n$ leads to 
$\skp{r_n,p_n} = 0$ by the definition of $\eta$, while $\skp{r_n,p_j}=0$, for 
$j= 1,\ldots, n-1$, by 
the induction hypothesis and the just proven orthogonality. This 
settles the Krylov step case.

For a rational step $n = 2k$, it follows by construction, cf.~\eqref{bab}, 
and the induction hypothesis for $n-1$ that 
$$\skp{\Op p_n, p_{n-1}} = 0.$$ 
Consider the equation for $p_n$: 
\[ \Op p_n + \alpha_k p_n = \rho_n p_{n-1} + r_{n-1}. \]
Taking the inner product with $\Op p_j$, for any $1 \leq j \leq n-2$, leads to 
\be\label{id2} \skp{\Op p_n,\Op p_j} + \alpha_k \skp{p_n,\Op p_j}  = 0. \ee
Indeed, this follows since $\skp{p_{n-1},\Op p_j} = 0$ by induction hypothesis, and, 
since $\Op p_j \subset \KRS^{n-1}$ by Theorem~\ref{th2} and \eqref{stru}, it follows that 
$$\skp{r_{n-1},\Op p_j} = \skp{r_{n-1}, \spa\{p_1,\ldots, p_{n-1}\}} = 0$$
 again by the induction  hypothesis. Thus, \eqref{id2} holds. 

Define the matrix/operator 
\[ \POP{n-1} = [p_1 \cdots p_{n-1}], \]
which maps $\R^{n-1}$ to $\KRS^{n-2}$. It follows from \eqref{id2} that 
for any coefficient vector $\vec{c} \in \R^{n-2}$
\be\label{id3} \skp{\Op p_{n-1} , \Op \POP{n-2} \vec{c}} + \alpha_k \skp{\Op p_n, \POP{n-2} \vec{c}} = 0.  \ee
By Theorem~\ref{th2} and \eqref{stru} again it follows that 
$\Op \POP{n-2} \vec{c} \in \KRS^{n-1}$, thus there exists a vector $\vg \in \R^{n-2}$,and $\vgm_{n-1} \in \R$ 
 with 
\be\label{id4} \Op \POP{n-2} \vec{c} = \POP{n-2} \vg + \vgm_{n-1} p_{n-1} \ee
since the $p_i$, $i=1,\ldots, n-1$, span $\KRS^{n-1}$. Hence, inserting this in \eqref{id3}
and noting $\skp{\Op p_{n},p_{n-1}}=0$ gives 
\be\label{id5}  \skp{\Op p_{n} , \POP{n-2} \vg + \alpha_k  \POP{n-2} \vec{c}} = 0 . \ee
Taking the inner product with $\Op \POP{n-2}$, we   get from  \eqref{id4}
\be  {\POP{n-2}}^T \Op  \Op \POP{n-2} \vec{c} = {\POP{n-2}}^T  \Op \POP{n-2} \vg, \ee 
where we again used the induction hypothesis $\skp{\Op p_j, p_{n-1}} = 0$ and
where ${\POP{n-1}}^T$ denotes the transposed operator. The matrix on the right-hand side 
is invertible (in fact it is diagonal by the induction hypothesis) and nonsingular since 
$\Op$ has no nullspace on $\KRS^{n-2}$. Thus, we can invert to find 
\[ \vg =  ({\POP{n-2}}^T  \Op \POP{n-2})^{-1} {\POP{n-2}}^T \Op  \Op \POP{n-2} \vec{c}, \]
and inserting this into \eqref{id5} gives 
 \be\label{id6} \skp{\Op p_{n} ,  [\POP{n-2} ({\POP{n-2}}^T  \Op \POP{n-2})^{-1} {\POP{n-2}}^T \Op  \Op \POP{n-2} \vec{c}+ \alpha_k \POP{n-2} \vec{c}]} = 0 . \ee
Now for any  $j\in \{1,\ldots,n-2\}$ we find a vector $\vec{c} \in \R^{n-2}$ such that 
\[ [({\POP{n-2}}^T  \Op \POP{n-2})^{-1} {\POP{n-2}}^T \Op  \Op \POP{n-2} + \alpha_k I ]\vec{c} = e_j, \]
Indeed, this identity can be rewritten as 
 \[  {\POP{n-2}}^T \Op  \Op \POP{n-2}+ \alpha_k ({\POP{n-2}}^T  \Op \POP{n-2}) ] \vec{c} = ({\POP{n-2}}^T  \Op \POP{n-2}) e_j.\]
The matrix on the left-hand side is a sum of positive semidefinite and positive define matrices, hence 
invertible, and  such a vector $\vec{c}$ exits. Inserting that into \eqref{id6} yields 
\[ \skp{\Op p_{n} , \POP{n-2} e_j}  = 0, \]
and thus $\skp{\Op p_{n}, p_{j}} = 0$ by the definition of $\POP{n-2}$. 
Hence the orthogonality for $p_n$ is proven. The orthogonality for $r_n$ in \eqref{verify} follows by 
the induction hypothesis, the update formula for $r_n$, the definition of $\eta$ and the 
just proven orthogonality relation for $p_n$. 
Together, \eqref{verify} is proven. 

The statement that $x_n$ is a least-squares solution in $\KRS^n$ is a now a simple consequence 
of the facts that $x_n \in \KRS^n$, that $p_i$ span $\KRS^n$,  and the orthogonality 
relations for the residual, which together implies $\skp{r_n,\KRS^n} = 0$. The latter is just the least-squares optimality condition in $\KRS^{n}$. 
We also note that in case of a breakdown the same conclusion is valid as well, except that the sequence 
of optimal $x_n$ saturates at the breakdown-index. If we extend the definition of the solution sequence $x_n$ in Algorithm~\ref{alg:ratcg} 
as the last one computed $x$ before breakdown, this still yields the least-squares solution even in this case. 
\end{proof}

\begin{remark}
Since the Lanczos iteration Algorithm~\ref{alg:lanz} also 
computes a least-squares solution in $\KRS^n$ and since by out assumptions the $x_n$ are uniquely defined, 
it follows that the sequences $x_n$ from  Algorithm~\ref{alg:lanz} and Algorithm~\ref{alg:ratcg}
are identical in exact arithmetic. (The corresponding $p_n$ might differ, though.)
\end{remark}

\subsection{Regularization and stopping rule}
The derived Algorithms~\ref{alg:lanz} and \ref{alg:ratcg} and also the aggregation method 
above shares some conceptional similarities with the CGNE method as they are 
(generalized) Krylov-space methods. In particular, all are nonlinear methods in the 
data~$y$. Note that the CGNE method without stopping rule is even  discontinuous 
in~$y$~\cite{EHN}, and we expect the same to be true for the stated method (including the aggregation method). However,  
it has been shown by Nemirovskii \cite{Nem1,Nem2} that the CGNE method with the discrepancy principle is a regularization method in the classical sense \cite{EHN,Hankebook}.
Thus, showing that the algorithms are regularization methods is most probably impossible without 
including a stopping criterion. 

 Therefore, by analogy, we include in the 
algorithms a discrepancy stopping rule and terminate the method for the first iteration index $n$, where 
\be\label{discprin} \|A x_n  - y\| \leq \tau \delta, \ee
where $\delta$ is the known noise level and $\tau >1$ is a tuning parameter, e.g., $\tau = 1.1$. 
This residual can be easily calculated after  $x_n$ is known, and the for-loop in the algorithms has 
to be terminated if the stopping criterion is satisfied. 
As mentioned above, a proof that this provides a regularization method is outside the scope of this work.

\section{Numerical Results}
We test the rational CG algorithm, Algorithm~\ref{alg:ratcg}, the rational Lanczos method, Algorithm~\ref{alg:lanz},
the aggregation method,  \eqref{aggmeth} (referred to as ``rational methods'')
and compare them with the classical conjugate gradient method 
for the normal equation CGNE as given, e.g., in \cite{EHN}. As simple test cases we used ten problems from 
the well-known Hansen Regularization Toolbox \cite{Hansentool}:
baart(2000), blur(60), deriv2(1000), gravity(1000), heat(1000), phillips(1000), shaw(1000), spikes(1000), wing(1000),
tomo(35) and their default exact solutions. In case of no noise  we use the default data, in case of 
a nonzero noiselevel, we added standard normal distributed random noise to the data
(by Matlab's randn command). In the latter case, all
algorithms were stopped using a discrepancy principle with $\tau = 1.01$. All problems have matrices 
with sizes of the order of $10^3 \times 10^3$. The ``blur'' and ``tomo'' examples have sparse matrices.

At first we test the performance for the noise-free case using  the exact toolbox data. The sequence of 
regularization parameters $\alpha_i$ was for all problems set as exponentially decreasing:
\be\label{alphachoice} \alpha_i = 10^{-i-1}. \ee
In Figure~\ref{fig1}, we display the least-squares residual $\|A x_n -y\|$ and the 
error $\|x_n -x_{exact}\|$ for the ``tomo'' problem and for the three proposed methods and the CGNE method. 
In the figure, 
the circles correspond to the values of the aggregation method with the same $\alpha_i$, i.e., 
the results for the aggregation method with a number of $k$ Tikhonov regularizations $x_{\alpha_i}$ 
are plotted at $n = 2k$, which is the index, where the rational CG and the rational Lanczos method need the same number of Tikhonov solves as the aggregation method. 
\begin{figure}
 \includegraphics[width=0.49\textwidth]{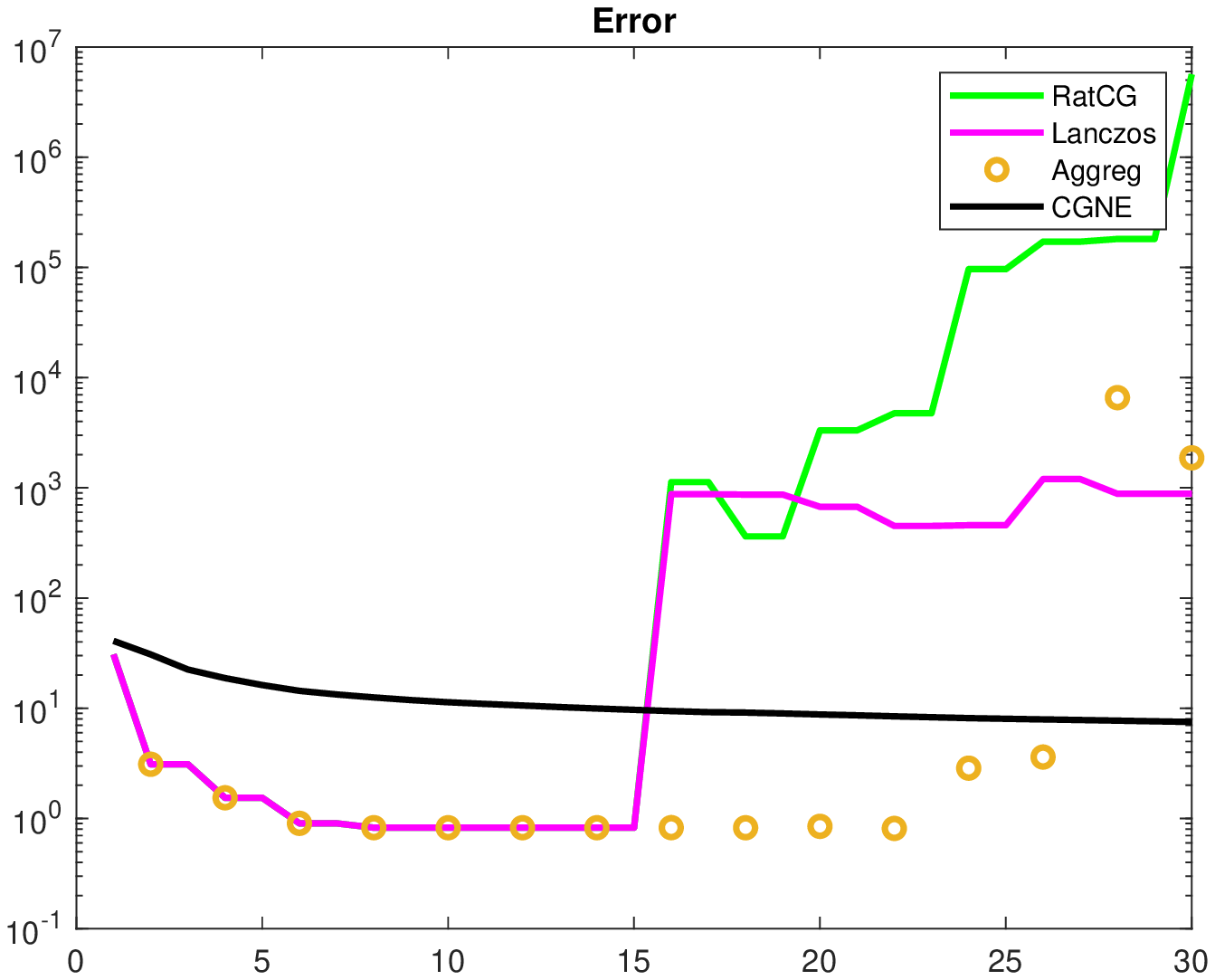}
  \includegraphics[width=0.49\textwidth]{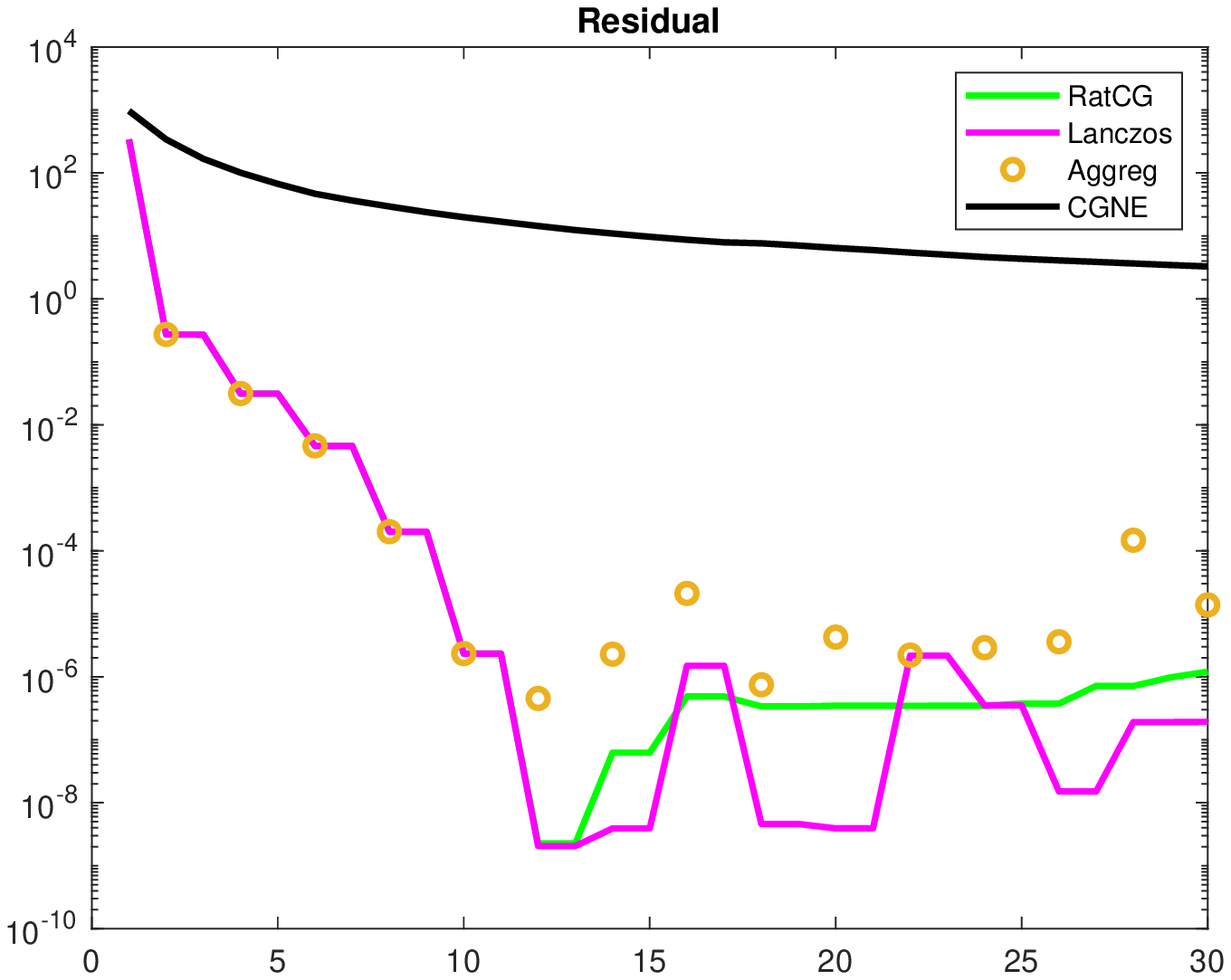}
  \caption{Logarithm of the error $\|x_n -x_{exact}\|$ and the residual $\|A x_n -y\|$ versus the iteration number for the rational CG method (Alg.~\ref{alg:ratcg},
  the Lanczos method (Alg.~\ref{alg:lanz}), the aggregation method \eqref{aggmeth}, and the 
  conjugate gradient method for the normal equation (CGNE). The circles represent 
  values for the aggregation method, plotted at the index $n = 2k$, where $k$ is the dimension 
  of the aggregation space (number of Tikhonov solutions). The noiselevel is $0$. 
  }\label{fig1}
\end{figure}

The general observation that seems to be true throughout is that the results for the 
rational CG method and the rational Lanczos method are in general identical (as predicted by the theory) as long 
as the $\alpha_i$ are not too small, but the method differ when we are in the realm of ill-conditioning (for  $\alpha$ too small). Note that the theory assumes exact arithmetic, which is no longer true when rounding error play a significant role. However, the difference of the two methods usually occurs  beyond a reasonable stopping rule. 
We also note that for all tested examples, the residual is smaller than that of the CGNE
as long as we are not in the ill-conditioning region. In many cases, the 
error for rational CG/Lanczos method agree with the corresponding ones 
of the aggregation method in the ``reasonable'' region. 
Note that the decay of the residual for 
CGNE is much slower, but of course, one has to take into account that each iteration of the 
method is of different complexity. 

In Table~\ref{tab1}, we present the result of the error $\|x_n -x_{exact}\|$, 
the computation time in seconds, and the number of iteration for the various methods 
and for the noise-free case. In order to find an appropriate $n$ in this case we first run the problem each up to a given maximal number of iteration and then choose that index $n$ 
where the error is the smallest (a kind of ``oracle'' stopping rule). The results are given in 
Table~\ref{tab1} together with the total running time (using Matlab's tic/toc command). 
The value of $5000$ for the CGNE method is the maximal number of used iteration.  

\begin{table}
\setlength{\tabcolsep}{3pt}   
\begin{center}
{\footnotesize
 \caption{Error $\|x_n -x_{exact}\|$, computation time (Time (s)), and number of iterations 
 for various problems form the Regularization Toolbox for the aggregation method \eqref{aggmeth},
  the Lanczos method (Alg.~\ref{alg:lanz}), the aggregation method, the rational CG
  method  (Alg.~\ref{alg:ratcg}), and the CGNE method. 5000 is $n_{max}$ for CGNE. 
  $\alpha_k$ chosen by \eqref{alphachoice}.}\label{tab1}
\begin{tabular}{cc||llll}
Problem&  & Aggreg. & Lanczos & RationalCG & CGNE \\ \hline 
baart & Error & 5.85E-02 &  6.58E-02 &  6.06E-02 & 2.47E-02 \\  
  & Time & 2.57 &  0.33 &  0.86 & 0.47 \\  
  & Iter. &     1 &      6 &     15 &   125 \\  \hline 
blur & Error & 3.13E-07 &  2.81E-12 &  2.88E-12 & 1.15E-13 \\  
  & Time & 4.84 &  2.49 &  2.32 & 0.11 \\  
  & Iter. &     1 &     13 &     13 &   575 \\    \hline 
deriv2 & Error & 7.41E-04 &  6.97E-04 &  9.17E-05 & 8.43E-03 \\  
  & Time & 0.71 &  0.50 &  0.27 & 1.24 \\  
  & Iter. &     1 &     77 &     34 &  5000 \\    \hline 
gravity & Error & 7.50E-03 &  4.48E-03 &  5.33E-03 & 3.70E-05 \\  
  & Time & 0.29 &  0.11 &  0.15 & 1.13 \\  
  & Iter. &     1 &     18 &     18 &  5000 \\    \hline 
heat & Error & 1.41E-02 &  1.18E-02 &  1.20E-02 & 1.40E-02 \\  
  & Time & 0.46 &  0.14 &  0.18 & 1.11 \\  
  & Iter. &     1 &     25 &     25 &  5000 \\    \hline 
phillips & Error & 1.89E-04 &  4.61E-05 &  5.71E-05 & 2.40E-05 \\  
  & Time & 0.28 &  0.11 &  0.13 & 1.12 \\  
  & Iter. &     1 &     17 &     16 &  5000 \\    \hline 
shaw & Error & 2.33E-01 &  2.31E-01 &  2.22E-01 & 2.86E-03 \\  
  & Time & 0.36 &  0.07 &  0.16 & 1.11 \\  
  & Iter. &     1 &     13 &     21 &  5000 \\    \hline 
spikes & Error & 2.60E+01 &  2.60E+01 &  2.60E+01 & 2.59E+01 \\  
  & Time & 0.24 &  0.10 &  0.14 & 1.01 \\  
  & Iter. &     1 &     17 &     17 &  5000 \\    \hline 
wing & Error & 1.90E-01 &  2.57E-01 &  1.90E-01 & 1.80E-01 \\  
  & Time & 0.32 &  0.04 &  0.06 & 0.03 \\  
  & Iter. &     1 &      5 &      7 &    83 \\    \hline 
tomo & Error & 8.39E-01 &  8.39E-01 &  8.39E-01 & 1.18E+00 \\  
  & Time & 0.63 &  0.22 &  0.29 & 0.74 \\  
  & Iter. &     1 &     13 &     12 &  5000 \\  
\end{tabular}
}
\end{center}
\end{table}

We observe that the rational CG and the Lanczos method perform roughly the same
and outperform the aggregation method both in terms of error and time. The CGNE method 
has a smaller error in 7 cases but requires less time in only 3 cases.

The next experiments concerns the case of nonzero noiselevel. The results 
for two noiselevels $\delta = 1\%, 0.1\%$ 
(using the discrepancy stopping rule) are given in Table~\ref{tab2}. 
\begin{table}
\setlength{\tabcolsep}{3pt}    
\begin{center}
{\scriptsize
 \caption{Error $\|x_n -x_{exact}\|$, computation time (Time (s)), and number of iterations 
 for various problems form the Regularization Toolbox for the aggregation method \eqref{aggmeth},
  the Lanczos method (Alg.~\ref{alg:lanz}), the rational CG
  method  (Alg.~\ref{alg:ratcg}), and the CGNE method and for two noiselevels $\delta = 1\%$ 
  and $0.1\%$. Stopping rule by the discrepancy principle with $\tau = 1.01$. }\label{tab2}
\begin{tabular}{c|cc||llll}
Problem & & & Aggreg. & Lanczos & Rat.CG & CG \\ \hline 
baart & $\delta=1\%$ & Error & 3.23E-01 &  3.20E-01 &  3.20E-01 & 2.09E-01 \\  
  &   & Time &  0.176 &  0.160 &  0.174 & 0.020 \\  
  &   & Iter. &       1 &      2 &      2 &     3 \\  
  & $\delta=0.1\%$ & Error & 2.06E-01 &  2.08E-01 &  2.08E-01 & 2.08E-01 \\  
  &   & Time &  0.353 &  0.165 &  0.181 & 0.020 \\  
  &   & Iter. &       1 &      3 &      3 &     3 \\   \hline 
blur & $\delta=1\%$ & Error & 2.39E+00 &  4.63E+00 &  4.63E+00 & 5.23E+00 \\  
  &   & Time &  0.733 &  0.327 &  0.342 & 0.004 \\  
  &   & Iter. &       1 &      2 &      2 &     9 \\  
  & $\delta=0.1\%$ & Error & 1.16E+00 &  1.09E+00 &  1.09E+00 & 1.36E+00 \\  
  &   & Time &  0.745 &  0.713 &  0.785 & 0.010 \\  
  &   & Iter. &       1 &      4 &      4 &    36 \\   \hline 
deriv2 & $\delta=1\%$ & Error & 1.66E-01 &  1.88E-01 &  1.88E-01 & 1.88E-01 \\  
  &   & Time &  0.115 &  0.036 &  0.037 & 0.005 \\  
  &   & Iter. &       1 &      4 &      4 &     4 \\  
  & $\delta=0.1\%$ & Error & 1.22E-01 &  1.20E-01 &  1.20E-01 & 1.22E-01 \\  
  &   & Time &  0.187 &  0.070 &  0.076 & 0.007 \\  
  &   & Iter. &       1 &      8 &      8 &     9 \\   \hline 
gravity & $\delta=1\%$ & Error & 7.19E-01 &  7.18E-01 &  7.18E-01 & 1.75E+00 \\  
  &   & Time &  0.037 &  0.030 &  0.031 & 0.003 \\  
  &   & Iter. &       1 &      2 &      2 &     4 \\  
  & $\delta=0.1\%$ & Error & 5.23E-01 &  5.22E-01 &  5.22E-01 & 8.06E-01 \\  
  &   & Time &  0.040 &  0.027 &  0.026 & 0.003 \\  
  &   & Iter. &       1 &      2 &      2 &     6 \\   \hline 
heat & $\delta=1\%$ & Error & 1.11E+00 &  9.92E-01 &  9.92E-01 & 1.30E+00 \\  
  &   & Time &  0.075 &  0.059 &  0.056 & 0.005 \\  
  &   & Iter. &       1 &      6 &      6 &     9 \\  
  & $\delta=0.1\%$ & Error & 4.16E-01 &  4.02E-01 &  4.02E-01 & 5.02E-01 \\  
  &   & Time &  0.150 &  0.068 &  0.074 & 0.008 \\  
  &   & Iter. &       1 &      8 &      8 &    18 \\   \hline 
phillips\hspace{-3mm} & $\delta=1\%$ & Error & 3.96E-02 &  3.94E-02 &  3.94E-02 & 7.98E-02 \\  
  &   & Time &  0.028 &  0.024 &  0.024 & 0.002 \\  
  &   & Iter. &       1 &      2 &      2 &     4 \\  
  & $\delta=0.1\%$ & Error & 3.77E-02 &  3.76E-02 &  3.76E-02 & 7.30E-02 \\  
  &   & Time &  0.027 &  0.025 &  0.024 & 0.002 \\  
  &   & Iter. &       1 &      2 &      2 &     4 \\   \hline 
shaw & $\delta=1\%$ & Error & 4.16E+00 &  4.15E+00 &  4.15E+00 & 5.25E+00 \\  
  &   & Time &  0.038 &  0.024 &  0.024 & 0.002 \\  
  &   & Iter. &       1 &      2 &      2 &     4 \\  
  & $\delta=0.1\%$ & Error & 1.63E+00 &  1.89E+00 &  1.89E+00 & 1.67E+00 \\  
  &   & Time &  0.110 &  0.038 &  0.041 & 0.003 \\  
  &   & Iter. &       1 &      5 &      5 &     6 \\   \hline 
spikes & $\delta=1\%$ & Error & 2.68E+01 &  2.68E+01 &  2.68E+01 & 2.67E+01 \\  
  &   & Time &  0.039 &  0.025 &  0.024 & 0.003 \\  
  &   & Iter. &       1 &      2 &      2 &     7 \\  
  & $\delta=0.1\%$ & Error & 2.62E+01 &  2.62E+01 &  2.62E+01 & 2.63E+01 \\  
  &   & Time &  0.047 &  0.025 &  0.025 & 0.004 \\  
  &   & Iter. &       1 &      2 &      2 &    13 \\   \hline 
wing & $\delta=1\%$ & Error & 3.48E-01 &  3.48E-01 &  3.48E-01 & 3.48E-01 \\  
  &   & Time &  0.062 &  0.028 &  0.027 & 0.002 \\  
  &   & Iter. &       1 &      2 &      2 &     2 \\  
  & $\delta=0.1\%$ & Error & 3.48E-01 &  3.48E-01 &  3.48E-01 & 3.48E-01 \\  
  &   & Time &  0.056 &  0.024 &  0.024 & 0.001 \\  
  &   & Iter. &       1 &      2 &      2 &     2 \\   \hline 
tomo & $\delta=1\%$ & Error & 6.31E+00 &  6.31E+00 &  6.31E+00 & 8.78E+00 \\  
  &   & Time &  0.105 &  0.078 &  0.083 & 0.006 \\  
  &   & Iter. &       1 &      2 &      2 &    16 \\  
  & $\delta=0.1\%$ & Error & 3.10E+00 &  3.10E+00 &  3.10E+00 & 5.03E+00 \\  
  &   & Time &  0.105 &  0.075 &  0.091 & 0.018 \\  
  &   & Iter. &       1 &      2 &      2 &    76 \\  
\end{tabular}
  }
  \end{center}
\end{table}
It can be seen that the Lanczos and rational CG method are about 
similar in behaviour. The running time 
of the CGNE method cannot be beaten by any of the rational methods.
The rational method need about 10 times more running time than CGNE. 
However, compared to  CGNE, we observe  that in a
majority of cases the proposed methods outperform 
the CGNE method in terms of the error. One reason for this is that 
the CGNE does not allow for fine-tuning of the regularization parameter 
(which is the iteration index in this case). An optimal stopping index for CGNE
would be ``in between'' two iterations, whatever that should mean.

The next figures concern the convergence rates of the discussed rational methods, i.e., aggregation, rational Lanczos, and 
the rational CG method, which are compares with the rates of the classical CGNE method. 
In Figure~\ref{fig2} we display 
the error $\|x_n -x^\dagger\|$ against various noiselevel $\delta$ on a log-log plot for all methods for four cases: 
First for the problem ``deriv2'' 
with the default true solution $x_{exact}$ and then again with a smoother solution, which is 
simply calculated by $x_{smooth} = A^*A x_{exact}$. This automatically implies a higher convergence 
rates for the later case because a higher source condition is satisfied. We do the same for the 
problem ``tomo'', i.e., default solution and smooth solution. The noise is generated again 
by samples from a standard normal 
random distribution.

The results for the two test cases for ``deriv2'' are displayed together on 
the left plot and that for ``tomo'' on the right. (The steeper slope corresponds to smoother solution). 
The reason for using smoother solution is to investigate whether the methods show a saturation in the 
convergence rates, which is known to happen for Tikhonov regularization. 
\begin{figure}
 \includegraphics[width=0.49\textwidth]{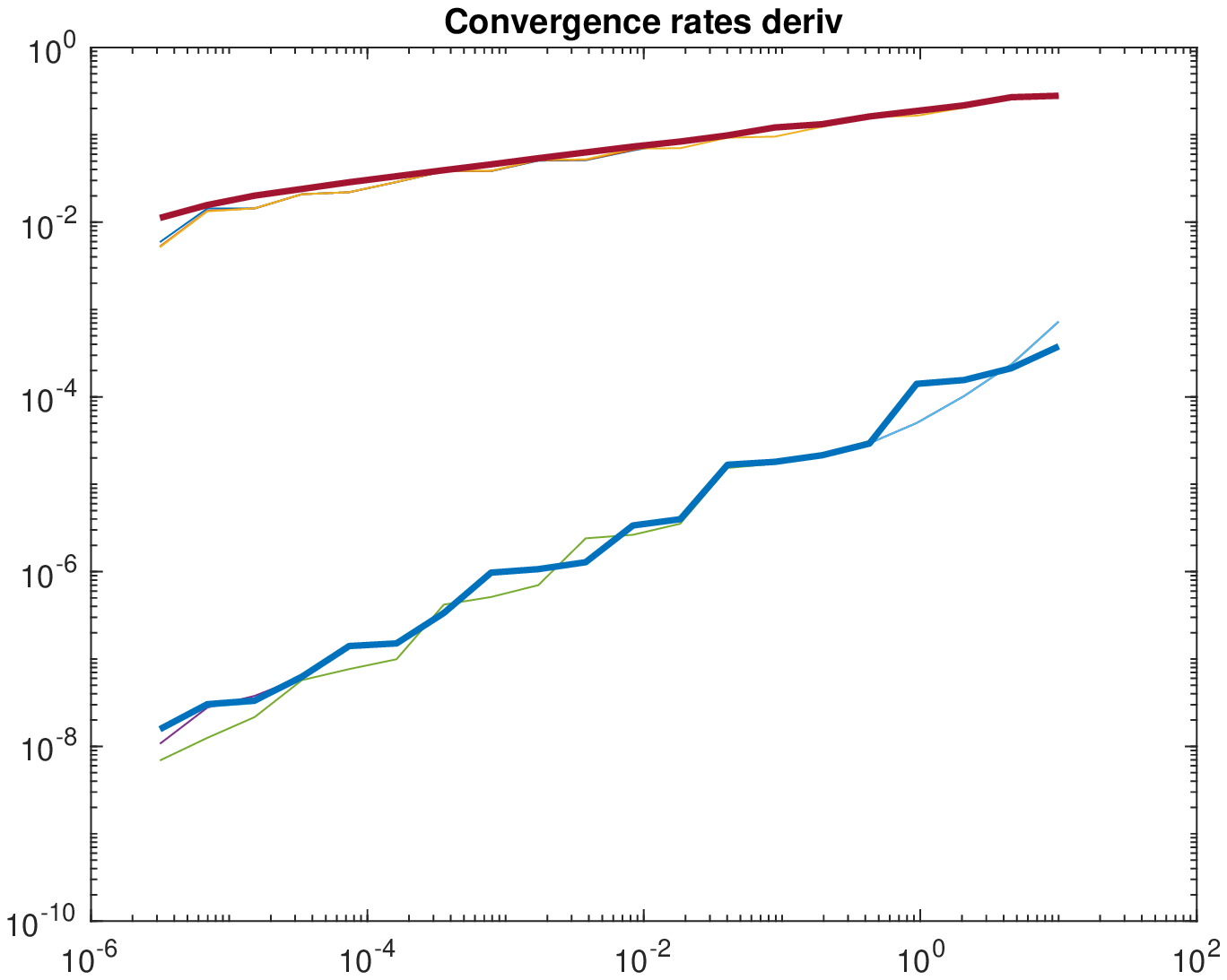}
  \includegraphics[width=0.49\textwidth]{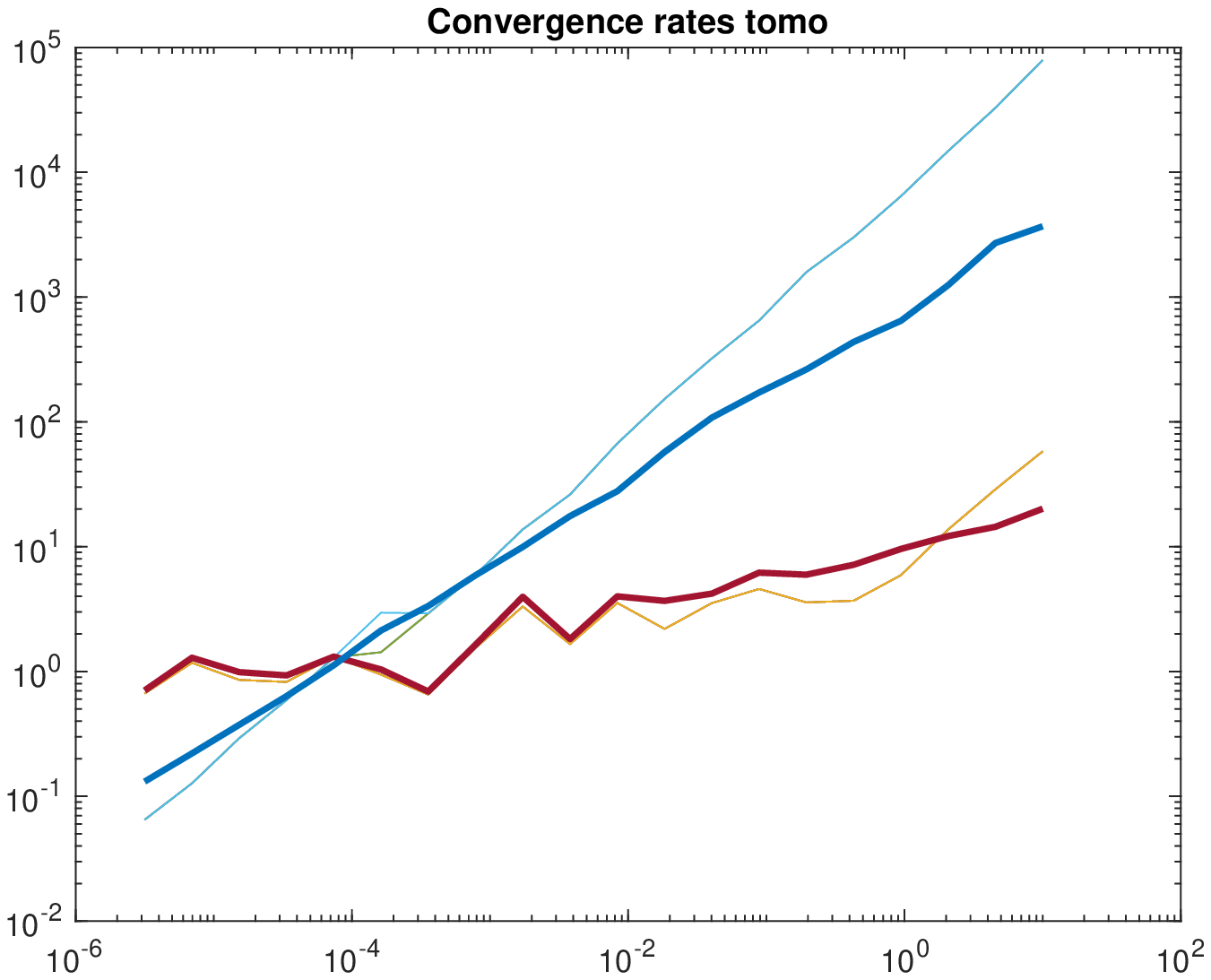}
  \caption{Convergence rates for the the aggregation method \eqref{aggmeth},
  the Lanczos method (Alg.~\ref{alg:lanz}), the rational CG
  method  (Alg.~\ref{alg:ratcg}), and the CGNE method for the  
  ``deriv2`` (left) and the ``tomo'' (right)  problems (overlayed plots). 
  Displayed is the error $\|x_n -x_{exact}\|$ vs. the noiselevel in a log-log plot for 
  two smoothness cases for each problem. (The steeper slopes correspond to the case of a smoother solution.)
  The error for CGNE is plotted by a thicker line as reference.
  }\label{fig2}
\end{figure}
For the ``deriv2'' problem, all  methods perform equally well, for the tomo problem this is true except for the 
CG method for the smooth solution, which has a slightly smaller slope (and thus a worse convergence rate). 
This figures should illustrate that all proposed methods show a similar (or even slightly better) rate than 
CGNE, which is known to achive the theoretically optimal-order rates. We observe that no saturation seems to happen, and we conjecture from the results that the rational methods are optimal-order 
method for all classical smoothness classes \cite{EHN}.

Finally, we tested the performance of the methods with respect to the choice of the sequence of regularization 
parameters. That is, we choose a geometrically decaying sequence of the form 
$\alpha_k = 0.1 q^{s -k}$, for various $q \in{2,4,6,8,10}$ and various starting values $s = -12, -10, \ldots, 8,10$. 
Thus, in the extreme cases, we start with a very large $\alpha$ or a very small one, and the $q$ controls the speed of the decay. We tested this for the ``tomo'' problem with $0.01 \%$ noise. Without details, we made the following 
observations: 
\begin{itemize}
 \item For most cases, the results were good, the speed of decay (choice of $q$) did not have much of  an influence. 
 \item The number of iterations is high (and the methods are slow) if we start with a very large $\alpha$, i.e., 
 far away from a reasonable good regularization parameter. 
 \item Starting with a too-small $\alpha$ (much below an ``optimal'' value) yields comparably 
 bad results (large error). 
       In this case, the aggregation, the  rational CG, and the Lanczos method  become unstable. However, 
       such a failure only happened for the extreme case $\alpha_k = 10^{-12 - k}$.  
 \item Starting with a too-large regularization parameter $\alpha$ leads to  stability 
 problems with the Lanczos method, for instance, when $\alpha_k = 10^{10 -k}$. This can be explained by the fact 
 that in such a case the Tikhonov inverse $(A^*A +\alpha I)^{-1} A^*y \sim \frac{1}{\alpha} A^*y  + O(\frac{1}{\alpha^2})$, 
is almost a scaled multiple of the first element in the Krylov space,
and thus the mixed Krylov space is close to 
 breakdown. The rational CG method behaved more robust in that respect. 
 \item As to be expected, if we start with an $\alpha$ that is already a good choice for classical Tikhonov regularization, then 
 the methods terminate after 2 iteration (i.e., after the first rational step) with good results. 
\end{itemize}

\subsubsection*{Further comments}
In terms of error we have verified an excellent performance of the rational methods, 
often even better than the CGNE method. The downside is, however, the additional computation 
time required. This, however, should not lead to  an a-priori refusal of these new methods. 
As soon as one admit Tikhonov regularization with a parameter choice search as useful methods, 
then the rational CG methods should be considered equally admissible, since it requires the 
same effort, has always a smaller residual, and, what is important, shows little sensitivity to the 
actually choice of regularization parameter, as long as we treat it as iterative method
coupled with a stopping rule.  

There is plenty of room for generalizing the methods and further investigations. We did not 
focus on tuning the rational methods; for instance, one can ease the Tikhonov inversions 
by invoking an a-priori factorization of the system matrix (for which even the rational Lanczos method
could be used). A stimulating piece of research would be to investigate the effect of 
an incomplete computation of the Tikhonov solutions. A extension of the method to the 
nonlinear case is highly interesting, but the modality is not obvious.

\section{Conclusion}
We derived the rational Lanczos and rational CG method for iteratively minimizing  linear least-squares problems over mixed rational Krylov spaces. We illustrate that these methods 
and the associated aggregation method perform equally well or better than the conjugate gradient method  for the normal equations in terms of the error. In terms of runtime, the CGNE method 
cannot be beaten but we did not attempt to improve this by further tuning. The main novelty is 
the rational CG method that requires only short recursions, with Tikhonov regularization in 
each second step, and nearly no additional memory requirements.

\end{document}